\documentclass[11pt]{article}
\usepackage[dvips]{graphicx}
\usepackage{bm}
\usepackage{amsmath,amsthm,amssymb,citesort}
\newtheorem{thm}{Theorem}[section]
\newtheorem{theorem}[thm]{Theorem}

\newtheorem{lemma}[thm]{Lemma}

\newtheorem{claim}[thm]{Claim}
\newtheorem{corollary}[thm]{Corollary}

\newtheorem{problem}{Problem}

\DeclareMathOperator{\rank}{rank}

\newcommand{\bp}{{\bm p}}
\newcommand{\bv}{{\bm v}}
\newcommand{\bq}{{\bm q}}
\newcommand{\br}{{\bm r}}
\newcommand{\bh}{{\bm h}}
\newcommand{\bal}{{\bm \alpha}}
\newcommand{\bmm}{{\bm m}}

\newcommand{\Bvector}[2]{\stackrel{#1}{\mathstrut #2}}

\begin{document}
\title{Generic Rigidity Matroids with Dilworth Truncations}
\author{Shin-ichi Tanigawa\thanks{Research Institute for Mathematical Sciences, Kyoto University. E-mail:{\tt tanigawa@kurims.kyoto-u.ac.jp}}}
\date{\today}
\maketitle
\begin{abstract}
We prove that the linear matroid that defines the generic rigidity of $d$-dimensional body-rod-bar frameworks
(i.e., structures consisting of disjoint bodies and rods mutually linked by bars)
can be obtained from  the union of ${d+1 \choose 2}$ copies of a graphic matroid by applying variants of Dilworth truncation operations $n_r$ times,
where $n_r$ denotes the number of rods.
This result leads to an alternative proof of Tay's combinatorial characterizations of
the generic rigidity of  rod-bar frameworks and that of identified body-hinge frameworks.
\end{abstract}

%{
%%
%\setcounter{footnote}{0}
%\def\thefootnote{\arabic{footnote}}
%%
%\footnotetext[1]{Supported by the project {\em  New Horizons in Computing},
%Grant-in-Aid for Scientific Research on Priority Areas, NEXT Japan.}
%\footnotetext[2]{Supported by Grant-in-Aid for JSPS Research Fellowships for Young Scientists.}
%Supported by JSPS Grant-in-Aid for Scientific Research on priority areas of New Horizons in Computing.}
%}
%

%\clearpage

\section{Introduction}

One of the main topics in rigidity theory is to reveal a
combinatorial characterization of the generic rigidity of frameworks.
Celebrated  Laman's theorem~\cite{laman:Rigidity:1970} asserts that
a 2-dimensional {\em bar-joint framework} (Fig.~\ref{fig:frameworks}(a)) is minimally rigid on a generic joint-configuration if and only if 
the graph $G=(V,E)$ obtained by regarding each joint as a vertex and each bar as an edge satisfies the following counting condition:
$|E|=2|V|-3$ and $|F|\leq 2|V(F)|-3$ for any nonempty $F\subseteq E$, where $V(F)$ denotes the set of vertices spanned by $F$.
However, in spite of exhausting efforts so far, the 3-dimensional counterpart has not been obtained yet (see, e.g.,\cite{jackson2006rank,Whitley:1997,whiteley:hand}).

A common strategy to deal with a difficult problem in graph theory is to restrict a graph class,
and several partial results are also known for the problem of characterizing 3-dimensional generic rigidity, 
for, e.g., triangulations~\cite{gluck,Whitley:1997}, bipartite graphs~\cite{whiteley:1984}, 
sparse graphs~\cite{jackson2006rank}, some minor closed classes~\cite{nevo2007}, the squares of graphs~\cite{molecular}.
In rigidity theory, it is also reasonable  to consider special types of structural models.
Tay~\cite{tay:84} considered a {\em body-bar framework} (Fig.~\ref{fig:frameworks}(b)) that consists of rigid bodies linked by bars.
He  proved that, if we represent the underlying graph by identifying each vertex with each body and each edge with each bar,
a body-bar framework is generically rigid in $\mathbb{R}^3$ if and only if 
the underlying graph contains six edge-disjoint spanning trees.
Tay~\cite{tay:89,tay1991linking} and Whiteley~\cite{whiteley:88} independently proved that, even for the {\em body-hinge} models (Fig.~\ref{fig:frameworks}(c)),
the same combinatorial characterization is true.
Specifically, a body-hinge framework is a structure consisting of rigid bodies connected by hinges.
Its underlying graph is represented by identifying each body with a vertex and each hinge with an edge.
In this setting, Tay-Whiteley's theorem asserts that a body-hinge framework is generically rigid in $\mathbb{R}^3$ if and only if
the graph obtained by duplicating each edge by five parallel copies contains six edge-disjoint spanning trees. 
Jackson and Jord{\'a}n~\cite{Jackson:07} further discuss the relation of generic rigidity of the {\em body-bar-hinge} model to the forest-packing problem in undirected graphs. 

Although it is barely mentioned, Tay's work was actually done in more general setting.
An {\em identified body-hinge framework} is a body-hinge framework in which each hinge allows to connect more than two bodies.
Historically, a combinatorial characterization of identified body-hinge frameworks was first conjectured by Tay and Whiteley in \cite{tay:whiteley:84},
and Tay affirmatively solved the conjecture in \cite{tay:89} as a by-product of his combinatorial characterization of {\em rod-bar frameworks}. 
A rod-bar framework  is a structure consisting of disjoint rods linked by bars in $\mathbb{R}^3$ (Fig.~\ref{fig:frameworks}(d)).
Each bar connects between two rods, and each rod is allowed to be incident to several distinct bars.
This structural model naturally comes up  from body-bar frameworks by regarding each rod as a degenerated 1-dimensional body.
%As observed in \cite[Section~7]{tay:89}, an identified body-hinge framework can be regarded as a special case of a rod-bar framework.
%Indeed, Tay's theorem on identified body-hinge frameworks is just a corollary of his combinatorial characterization of rod-bar frameworks 
%(see the statements given in Corollary~\ref{cor:rod_bar} and Corollary~\ref{coro:body_hinge}).

\begin{figure}[t]
\centering
\begin{minipage}{0.24\textwidth}
\centering
\includegraphics[width=0.75\textwidth]{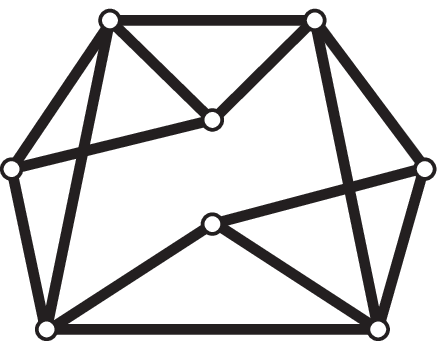}
\par 
(a)
\end{minipage}
\begin{minipage}{0.24\textwidth}
\centering
\includegraphics[width=0.9\textwidth]{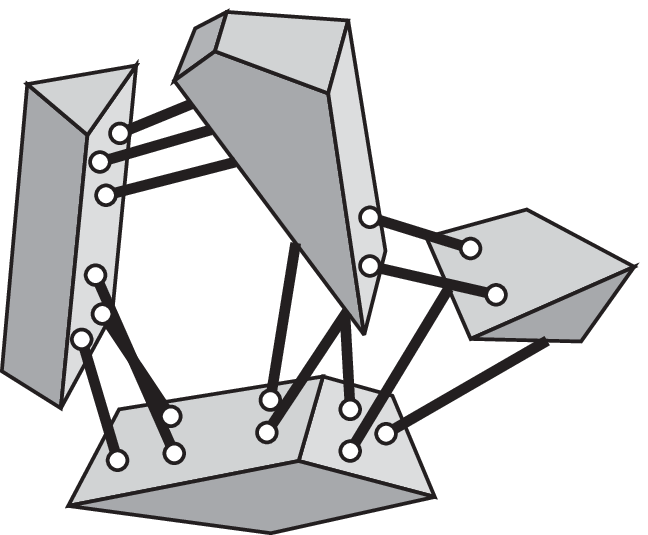}
\par 
(b)
\end{minipage}
\begin{minipage}{0.24\textwidth}
\centering
\includegraphics[width=0.9\textwidth]{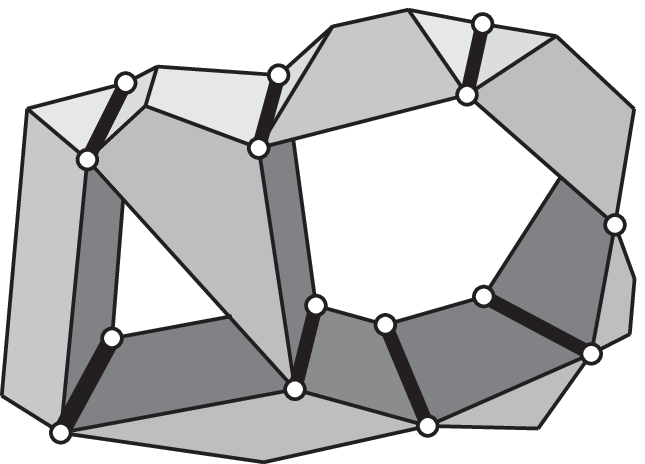}
\par 
(c)
\end{minipage}
\begin{minipage}{0.24\textwidth}
\centering
\includegraphics[width=0.9\textwidth]{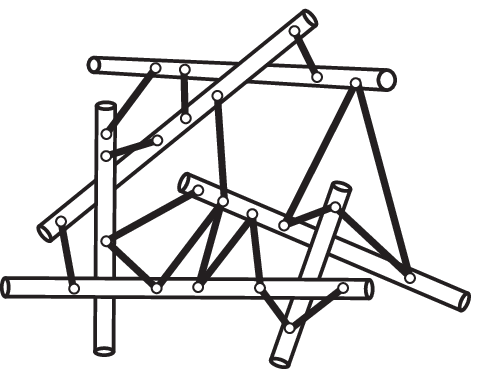}
\par 
(d)
\end{minipage}
\caption{(a)2-dimensional bar-joint framework, (b)body-bar framework, (c)body-hinge framework, and (d)rod-bar framework.}
\label{fig:frameworks}
\end{figure}

Unfortunately, Tay's proof is based on a Henneberg-type graph construction with intricate and long analysis 
(the combinatorial part now follows from the recent result by Frank and Szeg{\"o}~\cite{frank2003constructive}),
and the combinatorics behind rigidity of rod-bar frameworks has not been understood well.
To shed light on Tay's result, this paper provides a new proof of the combinatorial characterization of rod-bar frameworks. 

We actually cope with a more general structural model, {\em body-rod-bar frameworks},
and prove that the linear matroid defining its generic rigidity is equal to a counting matroid defined on the underlying graphs 
(Theorem~\ref{theorem:main} and Corollary~\ref{cor:body_rod_bar}).
Our proof technique is inspired by the idea of Lov{\'a}sz and Yemini given in \cite{lovasz:1982}.
They proved, as a new proof of Laman's theorem, that 
the linear matroid that defines the generic rigidity of 2-dimensional bar-joint frameworks
can be obtained from the union of two copies of a graphic matroid by Dilworth truncation. 
Roughly speaking, Dilworth truncation is an operation to construct a new linear matroid from old one,
by restricting the domain of entries of each vector to a generic hyperplane (see Subsection~\ref{subsec:truncation} for the definition).
The main difference between our situation and that of Lov{\'a}sz and Yemini is that we need to apply such truncation operations more than once 
(while they used it only once). 
Indeed, it is not trivial to keep up the representation of the resulting matroid 
when applying  Dilworth truncation operations several times, as each hyperplane must be inserted in ``generic'' position relative to the preceding hyperplanes.
We will overcome the difficulty by extending an idea of Lov{\'a}sz~\cite{lovasz:1977} 
so that each truncation is performed within a designated subspace.

A bar-joint framework can be considered as a body-bar framework consisting of $0$-dimensional bodies.
As combinatorial properties of body-bar frameworks with $3$-dimensional bodies are well understood~\cite{tay:84,white:whiteley:87,whiteley:88} in $\mathbb{R}^3$,
it is then natural to consider body-bar frameworks with $1$-dimensional bodies (i.e., rods) towards a combinatorial characterization of bar-joint frameworks.
Our proof explicitly describes how each $3$-dimensional body can be replaced by a $1$-dimensional body by the use of truncations.

%
%In general dimension $d$, a rod is defined as a $(d-2)$-dimensional affine space\footnote{Tay called a rod-bar framework as a $(d-2,2)$-framework.}. 
%In particular, in 2-dimension case rod-bar frameworks coincide with  bar-joint frameworks
%and hence our proof gives a new proof of Laman's theorem.  
%Also, we believe that our proof technique is so powerful that 
%it can be applied to more wide range of truncated matroids appeared in combinatorial geometry (see e.g.~\cite{Whitley:1997}).

The paper is organized as follows.
In Section~\ref{sec:flats}, we first review (poly)matroids induced by submodular functions, 
and then review two classical techniques proposed by Lov{\'a}sz~\cite{lovasz:1977}:
the first one shows how to obtain  a maximum matroid from a polymatroid defined by a family of flats in projective space,
and the second one is Dilworth truncation.
In Section~\ref{sec:body_bar}, we provide a proof of a combinatorial characterization of body-bar frameworks by Tay~\cite{tay:84}
from the view point of matroids of flat families (discussed in Section~\ref{sec:flats}).
Our main result is Section~\ref{sec:rod_bar},
where we prove a combinatorial characterization of body-rod-bar frameworks.
% based on the generalized Dilworth truncation proposed in Section~\ref{sec:extension}.
In Section~\ref{sec:body_hinge}, we will discuss identified body-hinge frameworks and several unsolved problems.
As another application of the Dilworth truncation, in Section~\ref{sec:direction}, we provide a direct proof of the combinatorial characterization of
$d$-dimensional direction-rigidity given by Whiteley~\cite[Theorem 8.2.2]{Whitley:1997}.
We believe that our proof technique is so powerful that 
it can be applied to more wide range of truncated matroids appeared in combinatorial geometry (see, e.g.,~\cite{Whitley:1997}).

We conclude introduction by listing some notation used throughout the paper.
For a vector space $W=\mathbb{R}^k$, let $\mathbb{P}(W)$ denote the projective space $\mathbb{P}^{k-1}$ associated with $W$.  
For a vector $\bv=(v^1,\dots, v^k)\in W$, the projective point associated with $\bv$ is denoted by $[\bv]=[v^1,\dots, v^k]\in \mathbb{P}(W)$.
For a flat $A$ in $\mathbb{P}(W)$, 
the {\em rank} of $A$ is defined by $\rank(A)=\dim W'$, where $W'$ is the linear subspace of $W$ associated with $A$. 
For a finite family ${\cal A}$ of flats, the {\em span} of ${\cal A}$ is denoted by $\overline{{\cal A}}$.
${\cal A}$ is called {\em disconnected} if there is a partition $\{{\cal A}_1,{\cal A}_2\}$ of ${\cal A}$ into nonempty subsets
such that $\rank (\overline{{\cal A}})=\sum_{i=1,2}\rank(\overline{{\cal A}_i})$ (equivalently, $\overline{{\cal A}_1}\cap \overline{{\cal A}_2}=\emptyset$).
%In such a  case, we write $\overline{\cal A}=\overline{{\cal A}_1}\oplus \overline{{\cal A}_2}$.
Otherwise ${\cal A}$ is said to be {\em connected}.  
(Note that a singleton set is connected.) 

We consider a finite graph $G=(V,E)$ that may contain parallel edges but no loop.
If $G$ has neither parallel edges nor a loop, $G$ is said to be {\em simple}.
We sometimes use notation $V(G)$ and $E(G)$ to denote the sets of vertices and edges of $G$, respectively.
For $v\in V$, let $\delta_G(v)$ be the set of edges incident to $v$ in $G$.
We say that $F\subseteq E$ {\em spans} $v\in V$ if $v$ is incident to some edge of $F$. 
For $F\subseteq E$, $V(F)$ denotes the set of vertices spanned by $F$.

\section{Preliminaries}
\label{sec:flats}
\subsection{Polymatroids}
\label{subsec:polymatroids}
Let $E$ be a finite set. A function $\mu:2^E\rightarrow \mathbb{R}$ is called {\em submodular}
if $\mu(X)+\mu(Y)\geq \mu(X\cup Y)+\mu(X\cap Y)$ for every $X,Y\subseteq E$.
$\mu$ is called {\em monotone} if $\mu(X)\leq \mu(Y)$ for every $X\subseteq Y$. 

Suppose $\mu:2^E\rightarrow \mathbb{Z}$ is an integer-valued function on $E$ satisfying $\mu(\emptyset)=0$.
The pair $(E,\mu)$ is called a {\em polymatroid}  if $\mu$ is monotone and submodular,
and $\mu$ is called the {\em rank function} of $(E,\mu)$.
It is particularly called a {\em matroid} if $\mu$ further satisfies  $\mu(e)\leq 1$ for every $e\in E$.
$F\subseteq E$ is called {\em independent} if $|F|=\mu(F)$, and a maximal independent set and a minimal dependent set are called a {\em base} and a {\em circuit}, respectively.
An element $e\in E$ is called a {\em coloop} if every base contains $e$.
 
\subsection{Submodular functions and induced polymatroids}
Suppose $\mu:2^E\rightarrow \mathbb{Z}$ is a monotone submodular function such that $\mu(F)\geq 0$ for every nonempty $F\subseteq E$ (but $f(\emptyset)<0$ is allowed).
We define $\hat{\mu}:2^E\rightarrow \mathbb{Z}$ by
\begin{equation}
\label{eq:g_hat}
\hat{\mu}(F)=\min\{\mbox{$\sum_{i=1}^k$}\mu(F_i)\} \qquad  (F\subseteq E)
\end{equation}
where the minimum is taken over all partitions $\{F_1,\dots,F_k\}$ of $F$ into nonempty subsets.
It is known that $\hat{\mu}$ is a monotone submodular function satisfying $\hat{\mu}(\emptyset)=0$ (see, e.g.,\cite[Chapter 48]{Schriver} or \cite{fujishige}),
and hence the pair $(E,\hat{\mu})$ forms a polymatroid.
It is also known that $\hat{\mu}$ is the unique largest 
among all monotone submodular functions $\mu'$ satisfying $0\leq \mu'(F)\leq \mu(F)$ for each $F\subseteq E$.

Edmonds and Rota~\cite{edmonds:1966} observed that a monotone submodular function $\mu:2^E\rightarrow \mathbb{Z}$
{\em induces} a matroid $(E,r_{\mu})$ on $E$,
where $F\subseteq E$ is independent if and only if $|F'|\leq \mu(F')$ for every nonempty $F'\subseteq F$ (see also \cite{pym}).
%We denote it by , denoted by ${\cal M}_g=(E,r_g)$,
Observe that this matroid takes the maximum rank among those satisfying $r_{\mu}(F)\leq \min\{\mu(F),|F|\}$ for every nonempty $F\subseteq E$,
and indeed the rank function $r_{\mu}$ can be written as
\begin{equation}
\label{eq:g_matroid}
r_{\mu}(F)=\min_{F_0\subseteq F}\{|F_0|+\hat{\mu}(F\setminus F_0)\} \qquad (F\subseteq E)
\end{equation}
(see, e.g.,\cite[Section 44.6a]{Schriver}).
Namely,
\begin{equation}
\label{eq:matroid_rank}
\mbox{$r_{\mu}(F)=\min\{|F_0|+\sum_{i=1}^k \mu(F_i)\}$} \qquad (F\subseteq E)
\end{equation}
where the minimum is taken over all partitions $\{F_0,F_1,\dots, F_k\}$ of $F$ such 
that $F_1,\dots, F_k$ are nonempty (and $F_0=\emptyset$ is allowed).
Geometric interpretations of these results will be discussed in the next two subsections.
More detailed descriptions on general (poly)matroids can be found in, e.g.,~\cite{Schriver,fujishige,oxley}.

\subsection{Generic matroids}
\label{subsec:points}
Let $E$ be a finite set.
We associate each element $e\in E$ with a flat $A_e$ in a real projective space,
and let ${\cal A}=\{A_e: e\in E\}$.
Also, for $F\subseteq E$, we denote $\{A_e\in {\cal A}: e\in F\}$ by ${\cal A}_F$. 
If we define a rank function $\rank_{\cal A}:2^E\rightarrow \mathbb{Z}$ by $\rank_{\cal A}(F)=\rank(\overline{{\cal A}_F})$ for $F\subseteq E$,
the pair $(E,\rank_{\cal A})$ forms a linear polymatroid, which is denoted by ${\cal PM}({\cal A})$.
A polymatroid turns out to be a matroid  by bounding the rank of each element by one.
Below, we review a geometric method for getting a maximum linear matroid from the linear polymatroid ${\cal PM}({\cal A})$.

We shall associate  a {\em representative point} $x_e\in A_e$ with each $A_e\in {\cal A}$.
Let us denote $\{x_e : e\in E\}$ by $X$.
The set $X$ of representative points is said to be in {\em generic position}
if, for every $X'\subseteq X$ and for every $x_e\in X'$,
\begin{equation}
\label{eq:point_generic}
x_e\in \overline{X'-x_e} \Rightarrow A_e\subseteq \overline{X'-x_e}.
\end{equation}
It is not difficult to see that, 
for any finite flat family ${\cal A}$, the set $X$ of representative points can be taken to be in generic position;
for any $x_e\in X$, $A_e\setminus \bigcup\{\overline{X'}: X'\subseteq X-x_e \text{ with } A_e\not\subset \overline{X'}\}$ forms a dense open subset of $A_e$;
hence, if $x_e\in \overline{X'}$ for some $X'\subseteq X-x_e$ with  $A_e\not\subset \overline{X'}$, 
then by continuously (and slightly) moving $x_e$ on $A_e$ it can avoid $\overline{X'}$ without creating a new violation 
for generic position.

For $F\subseteq E$,  the dimension of the linear subspace spanned by $\{x_e: e\in F\}$ 
is defined as the {\em rank} of $F$ (with respect to $X$), and we denote it 
by $\rank_X(F)$, i.e., $\rank_X(F)=\rank(\overline{\{x_e: e\in F\}})$.
The linear matroid $(E,\rank_X)$ is called a matroid associated with ${\cal A}$.
\begin{theorem}[Lov{\'a}sz~\cite{lovasz:1977}]
\label{theorem:flat_matroid}
Let ${\cal A}=\{A_e: e\in E\}$ be a finite family of flats, and $X$ be a set of representative points of ${\cal A}$ in generic position.
Then, 
\begin{equation}
\label{eq:flat_matroid_rank}
\rank_X(E)=\min_{F\subseteq E} \{|E\setminus F|+\rank(\overline{{\cal A}_{F}})\}.
\end{equation}
\end{theorem}
%\begin{proof}
%We write down the proof of Lov{\'a}sz since it is short and reveals the underlying structure.
%
%It is easy to see ``$\leq$'';
%for any $F\subseteq E$, we have 
%$\rank_X(E)\leq \rank_X(E\setminus F)+\rank_X(F)\leq |E\setminus F|+\rank(\overline{\{A_e\in{\cal A}: e\in F\}})$.
%
%Let us show the existence of $F\subseteq E$ that attains the equality.
%The proof is done by induction.
%If there exists an element $e\in E$ such that $\rank_X(E)=\rank_X(E-e)+1$.
%Then, by induction, we have a subset $F\subseteq E-e$ which attains the equality of (\ref{eq:flat_matroid_rank}) for $E-e$.
%This $F$ also attains the equality for $E$.
%
%Suppose every $e\in E$ satisfies  $\rank_X(E)=\rank_X(E-e)$.
%Then, $x_e\in \overline{X-x_e}$ for every $e\in E$, 
%and hence $A_e\subseteq \overline{X-x_e}\subseteq \overline{X}$ since $X$ is in generic position.
%This implies ${\cal A}\subseteq \overline{X}$, and we consequently obtain $\rank_X(E)=\rank(\overline{\cal A})=\rank(\overline{{\cal A}_E})$.
%\end{proof}
By restricting the argument to  $F\subseteq E$, we also have 
$\rank_X(F)=\min_{F'\subseteq F} \{|F\setminus F'|+\rank(\overline{{\cal A}_{F'}})\}$.
The rank of the linear matroid associated with ${\cal A}$ does depend on the choice of $X$.
However, Theorem~\ref{theorem:flat_matroid} implies that it attains the maximum and is invariant when $X$ is in generic position. 
(Notice that ``$\leq$'' direction of (\ref{eq:flat_matroid_rank}) holds even though $X$ is not in generic position;
For any $F\subseteq E$, $\rank_X(E)\leq \rank_X(E\setminus F)+\rank_X(F)\leq |E\setminus F|+\rank(\overline{\{A_e\in{\cal A}: e\in F\}})$.)
This motivates us to define the {\em generic} matroid.
The {\em generic matroid associated with ${\cal A}$}, denoted ${\cal M}({\cal A})$, 
is defined to be ${\cal M}({\cal A})=(E,\rank_X)$ with $X$ in generic position.

\subsection{Dilworth truncation}
\label{subsec:truncation}
Let ${\cal A}$ be a finite set of flats.
We now consider restricting flats of ${\cal A}$ to a generic hyperplane.
A hyperplane $H$ is called {\em generic relative to ${\cal A}$} if it satisfies the following condition\footnote{
Lov{\'a}sz claimed Theorem~\ref{theorem:truncation} with a much weaker assumption;
he defined that a hyperplane $H$ is  generic if,
for any subsets $X, Y$ and $Z$ of ${\cal A}$ satisfying 
$\overline{(X\cap H)\cup Y}\cap \overline{(X\cap H)\cup Z}\subseteq H$, we have $\overline{(X\cap H)\cup Y}\cap \overline{(X\cap H)\cup Z}\subseteq \overline{X\cap H}$. 
Theorem~\ref{theorem:truncation} however fails in this setting.
For example, suppose the underlying projective space is 3-dimensional, 
and ${\cal A}$ consists of three distinct hyperplanes $\{A_1,A_2,A_3\}$ such that $A_1\cap A_2=A_2\cap A_3=A_3\cap A_1$ is a line $l$. 
If we take $H$ as a hyperplane distinct from $A_i$ but containing $l$, 
$H$ satisfies the condition to be generic.
However, the left hand side of (\ref{eq:truncation}) is $\rank(\overline{\{A\cap H: A\in {\cal A}\}})= \rank(l)= 2$ 
while the right hand side is equal to $\rank(\overline{\{A_1,A_2,A_3\}}) -1=3$.};
for any $A_1,A_2\in {\cal A}$ and any ${\cal F}\subseteq \{A\cap H: A\in {\cal A}\}$, 
\begin{equation}
\label{eq:hyperplane_generic}
\overline{(A_1\cap H)\cup {\cal F}}\cap \overline{(A_2\cap H)\cup {\cal F}}\neq \overline{{\cal F}} 
\quad \Rightarrow \quad \overline{A_1\cup {\cal F}}\cap \overline{A_2\cup {\cal F}}\not\subset H.
\end{equation}
Although the detail is omitted, it can be verified  that almost all hyperplanes are generic relative to ${\cal A}$.
For a family ${\cal A}$ of flats and a hyperplane $H$, we shall abbreviate
$\{A\cap H : A\in {\cal A}\}$ as ${\cal A}\cap H$.
The following result is also done by Lov{\'a}sz~\cite{lovasz:1977}.
\begin{theorem}[Lov{\'a}sz~\cite{lovasz:1977}]
\label{theorem:truncation}
Let ${\cal A}$ be a finite family of flats in a real projective space, and $H$ be a generic hyperplane relative to ${\cal A}$.
Then, 
\begin{equation}
\label{eq:truncation}
\rank(\overline{{\cal A}\cap H}) = \min\{\mbox{$\sum_{i=1}^k$} (\rank(\overline{{\cal A}_i})-1)\},
\end{equation}
where the minimum is taken over all partitions $\{{\cal A}_1,\dots, {\cal A}_k\}$ of ${\cal A}$ into nonempty subsets. 
\end{theorem}
This operation (of restricting flats to a generic hyperplane) is referred to as {\em Dilworth truncation}.
Indeed, as noted in \cite{Schriver}, Theorem~\ref{theorem:flat_matroid} and Theorem~\ref{theorem:truncation} provide geometric interpretations of the formulae (\ref{eq:g_hat}) and (\ref{eq:matroid_rank}) for linear polymatroids.

The same result was also obtained by Mason~\cite{Mason:1981,Mason:1976} from the view point of combinatorial geometry (projective matroids).
The papers of Mason~\cite{Mason:1981,Mason:1976} include examples of Dilworth truncation.

\subsection{$M$-connectivity and $P$-connectivity}
\label{sec:connectivity}
Let ${\cal M}=(E,r)$ be a matroid on a finite set $E$ with the rank function $r$.
A subset $F\subseteq E$ is called {\em $M$-connected} if, for any pair $e,e'\in F$, $F$ has a circuit of ${\cal M}$ that contains $e$ and $e'$. 
For simplicity of the description, a singleton $\{e\}$ is also considered as an $M$-connected set.
A maximal $M$-connected set is called an {\em $M$-connected component}.
%A connected component of ${\cal M}_f(G)$ (in the sence of matroid theory) is called a {\em $f$-connected component}.
It is well know that the union of two $M$-connected sets is $M$-connected if their intersection is nonempty, and thus $E$ is uniquely partitioned into $f$-connected components $E_1,\dots, E_k$ (see, e.g.,\cite[Chapter~4]{oxley}). 
Since there is no circuit intersecting two components, we have $r(E)=\sum_{i=1}^kr(E_i)$.
Alternatively, we can use it for the definition of $M$-connectivity: 
$F\subseteq E$ is $M$-connected if and only if 
there no partition $\{F_1,\dots, F_k\}$ of $F$ into at least two nonempty subsets such that $r(F)=\sum_{i=1}^k r(F_i)$.

%An $M$-connected set is called {\em trivial} if it is singleton; otherwise nontrivial. 
%We remark that $\{e\}$ is a trivial $M$-connected component 
%if and only if $e$ is either a {\em loop} (i.e., no base contains $e$) 
%or a {\em coloop} (i.e., every base contains $e$). 

The concept of connectivity can be extended to polymatroids.
Let ${\cal PM}=(E,\mu)$ be a polymatroid on a finite set $E$.
Then, $F\subseteq E$ is said to be {\em $P$-connected} 
if there is no partition $\{F_1,\dots, F_k\}$ of $F$ into at least two nonempty subsets such that $\mu(F)=\sum_{i=1}^k \mu(F_i)$.
A maximal $P$-connected set is called a {\em $P$-connected component}.
The union of two $P$-connected sets is $P$-connected if their intersection is nonempty, and thus $E$ is uniquely partitioned into $P$-connected components.
If we consider linear polymatroids, the concept of $P$-connectivity coincides with 
the connectivity of flats we introduced in the introduction.

A $P$-connected set (and similarly, an $M$-connected set) is called {\em trivial} 
if it is singleton; otherwise nontrivial.  

\section{Body-bar Frameworks}
\label{sec:body_bar}
A body-bar framework is a structure consisting of rigid bodies linked by bars (Figure~\ref{fig:frameworks}(b)).
The generic rigidity of body-bar frameworks is characterized by Tay~\cite{tay:84} (and a simpler proof was given by Whiteley~\cite{whiteley:88}).
In this section, we shall present a proof of this characterization from the viewpoint of matroids of flat families.
In the subsequent sections, $d$ denotes the dimension of frameworks, and let $D={d+1 \choose 2}$.

\subsection{Union of copies of graphic matroid}
\label{subsec:graphic}
We first review the union of copies of graphic matroid to which Tay related the generic rigidity matroid 
in the body-bar model.

\subsubsection{Graphic matroid}
Let $G=(V,E)$ be a finite undirected graph.
We denote  the {\em graphic matroid} of $G$ by ${\cal G}(G)$,
that is, the matroid induced by the monotone submodular function $g:2^E\rightarrow \mathbb{Z}$ defined by
$g(F)=|V(F)|-1$ for $F\subseteq E$.
Namely, $F\subseteq E$ is independent in ${\cal G}(G)$ if and only if $|F|\leq |V(F)|-1$ for nonempty $F\subseteq E$, and equivalently $F$ is a forest.

Let $I(G)=[a_{ij}]$ be the incidence matrix of a digraph obtained from $G$ by arbitrary assigning a
direction to each edge, i.e,
\[
 a_{ij}=\begin{cases}
1 & \text{if vertex $v_j$ is the tail of arc $e_i$} \\
-1 & \text{if vertex $v_j$ is the head of arc $e_i$} \\
0 & \text{otherwise}.
\end{cases}
\]  
It is well known that ${\cal G}(G)$ is linear
as it is represented by the row matroid of  $I(G)$.

\subsubsection{Graphic matroid union}
For a matroid ${\cal M}=(E,{\cal I})$ with a collection  ${\cal I}$ of independent sets,
the union of $D$ independent sets, i.e., $\{I_1\cup\dots \cup I_D :  I_i\in {\cal I}, i=1,\dots, D\}$, again 
forms the collection of independent sets of a matroid.
This matroid is called the {\em union} of $D$ copies of ${\cal M}$.
In the union of $D$ copies of the graphic matroid, denoted $D{\cal G}(G)$, $F\subseteq E$ is independent
if and only if $F$ can be partitioned into $D$ edge-disjoint forests.
$D{\cal G}(G)$ is indeed the matroid induced by the monotone submodular function $Dg:=D(|V(\cdot)|-1)$ defined on $E$~\cite{nash1964decomposition}.
This implies that $E$ can be partitioned into $D$ edge-disjoint spanning trees if and only if $|E|=D(|V|-1)$ and
$|F|\leq D(|V(F)|-1)$ for any nonempty $F\subseteq E$.

It is also known that $D{\cal G}(G)$ can be represented as a row vector matroid by introducing indeterminates.
For each integer $k$ with $1\leq k\leq D$, let $I^k=[a_{ij}^k]$ be a $|E|\times |V|$-matrix defined by 
\begin{equation*}
 a_{ij}^k=\begin{cases}
\alpha_{e_i}^k & \text{if vertex $v_j$ is the tail of arc $e_i$} \\
-\alpha_{e_i}^k & \text{if vertex $v_j$ is the head of arc $e_i$} \\
0 & \text{otherwise},
\end{cases}
\end{equation*}
where $\alpha_{e}^k$'s are algebraically independent indeterminates over $\mathbb{Q}$.
Denote the $|E|\times D|V|$-matrix $[I^1|I^2|\dots |I^D]$ by $DI(G)$.
Then, $D{\cal G}(G)$ is represented by $DI(G)$ (see, e.g.,~\cite{Mason:1981,whiteley:88}).

This representation gives us another way to look at $D{\cal G}(G)$.
We associate a $D$-dimensional vector space $V_u=\mathbb{R}^D$ with each vertex $u$ in the subsequent discussion,
and $V_V$ denotes the direct product of $V_u$ for all $u\in V$.
In $DI(G)$, the row associated with an edge $e=uv$ is represented by
\begin{equation}
\label{eq:e_vector}
(\Bvector{}{0},\Bvector{\cdots}{\cdots}, \Bvector{}{0},  
\Bvector{u}{\overbrace{\alpha_{e}^1,  \dots,  \alpha_{e}^D}},  \Bvector{}{0}, \Bvector{\cdots}{\cdots}, 
\Bvector{}{0}, 
\Bvector{v}{\overbrace{-\alpha_{e}^1, \dots, -\alpha_{e}^D}}, \Bvector{}{0}, \Bvector{\cdots}{\cdots}, \Bvector{}{0}),
\end{equation}
where we changed the column ordering so that  the entries associated with each vertex form a block 
(and throughout the subsequent discussions we will refer to this ordering).
When looking $\alpha_{e}^1,\dots, \alpha_e^D$ as independent parameters in $\mathbb{R}$,
the space spanned by vectors (\ref{eq:e_vector}) form a $D$-dimensional vector space contained in 
$V_u\times V_v$.
%\begin{equation}
%\label{eq:e_space}
%A_e=\{ 
%(\Bvector{}{0} \ \Bvector{\cdots}{\cdots} \ \Bvector{}{0} \ \ 
%\Bvector{u}{\alpha_{e}^1 \ \dots \ \alpha_{e}^D}\ \ \Bvector{\cdots}{0} \ \ \Bvector{\cdots}{\cdots} \ \ 
%\Bvector{\cdots}{0}\ \ 
%\Bvector{v}{-\alpha_{e}^1 \ \dots \ -\alpha_{e}^D}\ \ \Bvector{}{0} \ \ \Bvector{\cdots}{\cdots}\ \ \Bvector{}{0}):
%(\alpha_{e}^1,\dots, \alpha_{e}^D)\in \mathbb{R}^D\}.
%\end{equation}
We can identify this $D$-dimensional vector space with a $(D-1)$-dimensional flat in $\mathbb{P}(V_V)$.
We denote this flat by $A_e$ and let  ${\cal A}:=\{A_e :  e\in E\}$.
Then,  $D{\cal G}(G)$ can be considered as the generic matroid ${\cal M}({\cal A})$ associated with ${\cal A}$. 

\subsection{Generic body-bar matroids}
\label{subsec:body_bar}
\subsubsection{Pl{\"u}cker coordinates}
Throughout the paper, let $W=\mathbb{R}^{d+1}$.
For simplicity,  we shall use the standard basis $e_1,\dots, e_{d+1}$ of $W=\mathbb{R}^{d+1}$ and use the dot product as an inner product.
Also $W$ is identified with  its dual.

Recall that the exterior product $\bigwedge^k W$ of degree $k$ is a ${d+1 \choose k}$-dimensional vector space
and can be naturally identified with $\mathbb{R}^{d+1\choose 2}$ by associating 
$e_{i_1}\wedge \dots \wedge e_{i_k}$ with an element of the standard basis of $\mathbb{R}^{d+1\choose 2}$ 
for each $1\leq i_1<\dots<i_k<d+1$. 
In particular, $\bigwedge^2 W=\mathbb{R}^D$.
%An element of $\bigwedge^k W$ is called  {\em decomposable} if it can be described as the exterior product of $k$ elements in $W$. 
%We denote the exterior product $v_1\wedge v_2 \wedge \dots \wedge v_k$ simlpy by $v_{1,2,\dots,k}$.

The collection of $k$-dimensional subspaces in $W$ is called the {\em Grassmannian}, denoted $Gr(k,W)$.
The {\em Pl{\"u}cker embedding} $p:Gr(k,W)\rightarrow \mathbb{P}(\bigwedge^k W)$ is a bijection
between $k$-dimensional vector spaces $X\in Gr(k,W)$ and 
projective equivalence classes $[v_1\wedge \dots \wedge v_k]\in \mathbb{P}(\bigwedge^k W)$ of decomposable elements,
where $\{v_1,\dots, v_k\}$ is a basis of $X$.
In the subsequent discussions, we shall identify $Gr(k,W)$ and its image of the Pl{\"u}cker embedding,
and regard $Gr(k,W)$ as a subset of $\mathbb{P}(\bigwedge^k W)$.   
%The homogenious corrdinate of $[v_1\wedge \dots \wedge v_k]\in \mathbb{P}(\bigwedge^k W)$ for a fixed basis $e_1,\dots, e_{d+1}\in W$
%is called the {\em Pl{\"u}cker coordinate} of $X$.

%Suppose $W=\mathbb{R}^{d+1}$.
It is well-known that each point of $Gr(k,W)$ can be coordinatized by the so-called {\em Pl{\"u}cker coordinate} once we fix  a basis of $W$. 
If a basis $\{v_1,\dots, v_k\}$ of $X\in Gr(k,W)$ is represented by $v_i=\sum_{j=1}^{d+1}p_{ij}e_j$ 
with the $k\times (d+1)$-matrix $P=[p_{ij}]$,
then we have 
\[
v_1\wedge \dots \wedge v_k=\mbox{$\sum_{i_1<\cdots<i_k}$} \det P_{i_1,\dots, i_k} e_{i_1}\wedge \dots \wedge e_{i_k}, 
\]
where $P_{i_1,\dots, i_k}$ is the $k\times k$-submatrix of $P$ consisting of $i_j$-th columns.
Let us simply denote $p_{i_1,\dots, i_k}=\det P_{i_1,\dots, i_k}$.
The ratio of $p_{i_1,\dots, i_k}$'s for $1\leq i_1< \dots < i_k\leq d+1$ is called {\em the Pl{\"u}cker coordinate} of $X$. 
%Therefore, it is sometimes useful to index coordinates of an element of $\bigwedge^k W$ by $k$-tuples, $i_1,\dots, i_k$ with $1\leq i_1< \dots < i_k\leq d+1$.

It is well known that $[p_{i,j}]_{1\leq i<j\leq d+1}\in \mathbb{P}(\bigwedge^2 W)$ is in $Gr(2,W)$
if and only if $p_{i,j}p_{k,l}-p_{i,k}p_{j,l}+p_{i,l}p_{j,k}=0$ for $1\leq i<j<k<l\leq d+1$,
and $Gr(2,W)$ is an irreducible quadratic variety (see, e.g.,\cite{Hassett200705}).
In particular, if $d=3$, 
$Gr(2,W)$ is a non-singular quadratic variety written by
\begin{equation}
\label{eq:grassmaniann}
 \{[p_{i,j}]_{1\leq i<j\leq 4}\in \mathbb{P}(\mbox{$\bigwedge^2 W$})  : p_{1,2}p_{3,4}-p_{1,3}p_{2,4}+p_{1,4}p_{2,3}=0\}.
\end{equation}
Through the one-to-one correspondence between a $k$-dimensional linear subspace and its orthogonal complement,  
$Gr(d-1,W)$ is also an irreducible quadratic variety in $\mathbb{P}(\bigwedge^{d-1} W)$ described in the same form as $Gr(2,W)$.

Let us define a product 
$\langle \cdot , \cdot \rangle\colon\bigwedge^k W\times \bigwedge^{d+1-k}W\rightarrow \mathbb{R}$ by 
\[
 \langle \bp, \bq\rangle =\sum_{i_1<\dots <i_k} (-1)^{i_1+\dots +i_k}p_{i_1,\dots, i_k} q_{j_1,\dots, j_{d+1-k}}
\]
for $\bp=[p_{i_1\dots i_k}] \in \bigwedge^k W$  and $\bq=[q_{i_1\dots i_{d+1-k}}]\in \bigwedge^{d+1-k} W$,
where $j_1,\dots, j_{d+1-k}$ are the complement of $i_1,\dots, i_k$ in $[d+1]$ with $j_1<\dots <j_{d+1-k}$.
For example, for $d=3$ and $k=2$, we have 
$\langle \bp,\bq\rangle=p_{1,2}q_{3,4}-p_{1,3}q_{2,4}+p_{1,4}q_{2,3}+p_{2,3}q_{1,4}-p_{2,4}q_{1,3}+p_{3,4}q_{1,2}$.
In general, it has the following useful property: 
a $k$-dimensional linear subspace $X$ and a $(d+1-k)$-dimensional linear subspace $Y$ have a nonzero intersection if and only if the corresponding Pl{\"u}cker coordinates $[\bp]$ and $[\bq]$
satisfy $\langle \bp, \bq \rangle=0$. 
This is because that, if $\bp$ and $\bq$ are decomposable, 
then $\langle \bp, \bq\rangle$ is the determinant of a square matrix obtained 
by aligning composition elements of $\bp$ and $\bq$.

This product can be seen as a dot product in $\mathbb{R}^{d+1\choose k}$ through the so-called Hodge star-operator.
The Hodge star-operator is a linear operation $\ast\colon \bigwedge^k W\rightarrow \bigwedge^{d+1-k} W$ defined by 
\[
 \ast(e_{i_1}\wedge \dots \wedge e_{i_k})={\rm sign}(\sigma)  e_{j_1}\wedge \dots \wedge e_{j_{d+1-k}},
\]
where $j_1,\dots, j_{d+1-k}$ are the complement of $i_1,\dots, i_k$ in $[d+1]$  
and ${\rm sign}(\sigma)$ denotes the sign of the permutation 
$\sigma=\begin{pmatrix}i_1 & \dots & i_k & j_1 & \dots & j_{d+1-k} \\ 1 & \dots & k & k+1 & \dots & d+1 \end{pmatrix}$.
%In our situation, the Hodge star-operator just returns a $(d+1-k)$-extensor of the orthogonal complement.  
%In particular,  $\langle \alpha, \beta \rangle=0$ if and only if $\alpha \cdot \ast \beta=0$.
For example, if $d=3$ and $k=2$, $\ast \bq=(q_{3,4},-q_{2,4},q_{2,3}, q_{1,4},-q_{1,3},q_{1,2})$ for $\bq=(q_{1,2},q_{1,3},q_{1,4}, q_{2,3},q_{2,4},q_{3,4})$.

By identifying $\bigwedge^k W$ with $\bigwedge^{d+1-k} W$ through $\ast$ and identifying $\bigwedge^k W$ with  $\mathbb{R}^{d+1\choose k}$ as above, 
we may consider $\langle \cdot,\cdot\rangle$ as a dot product in $\mathbb{R}^{d+1\choose k}$.
In this way we can simply consider a dot product between $\bigwedge^k W$ and $\bigwedge^{d+1-k}$, 
where  $\bp \cdot \bq=0$  if and only if
$X\cap Y=\{0\}$, for a $k$-dimensional linear subspace $X$ and a $(d+1-k)$-dimensional linear subspace $Y$ with the Pl{\"u}cker coordinates $[\bp]$ and $[\bq]$.

For general treatments of these operations, see e.g.~\cite{barnabei1985exterior,HodgePedoe}.

%
%It is well known that a point of $\bigwedge^2 W$ is in $Gr(2,W)$ (namely, it is decomposable)
%if and only if it vanishes a system of homgenious quadratic poynomials.
%Namely, $Gr(2,W)$ is a projective variety whose ideal is generated by quadratic polynomials, so-called {\em Pl{\"u}cker relations}.
%In particular, for $d=3$, 
%$Gr(2,W)$ is the nondegenerate quadratic variety (i.e., the associated matrix describing the quadratic form is non-singular) described by
%\begin{equation}
%\label{eq:grassmaniann}
% \{[p_i]_{1\leq i\leq 6}\in \mathbb{P}(\mbox{$\bigwedge^2$} W) : p_{1}p_{4}-p_{2}p_{5}+p_{3}p_{6}=0\}.
%\end{equation}

%********
\subsubsection{Body-bar frameworks}
We shall use following conventional notation to denote body-bar frameworks and to describe infinitesimal motions.
A {\em  body-bar framework}  is a pair $(G,\bq)$, where 
\begin{itemize}
\item $G=(V,E)$ is a graph;
\item $\bq$ is a mapping called a {\em bar-configuration}:
\begin{align*}
\bq:E&\rightarrow Gr(2,W)\subseteq \mbox{$\mathbb{P}(\bigwedge^{2}W)$} \\
e &\mapsto [\bq_e]=[q_e^1,\dots,q_e^D].
\end{align*}
\end{itemize}
Namely, a line $\bq(e)$ associated with $e=uv$ represents a bar connecting between two bodies associated with $u$ and $v$.
An {\em infinitesimal motion} of $(G,\bq)$ is a mapping $\bmm:V\rightarrow \bigwedge^{d-1} W$ satisfying
\begin{equation}
\label{eq:bar_constraint}
\bq_e\cdot (\bmm(u)-\bmm(v))=0 \qquad \text{for all } e=uv\in E.
\end{equation}
This definition is essentially the same as the conventional one used in the bar-joint model,
in the sense that it  requires the orthogonality of the direction of a bar and the difference of 
infinitesimal motions assigned to the adjacent bodies. 
A detailed geometric meaning of  (\ref{eq:bar_constraint}) is explained in Appendix A. (Detailed description can be also
found in e.g., \cite{tay1991linking,white:whiteley:87,White94,Jackson:07}.)

The set of infinitesimal motions forms a $D|V|$-dimensional vector space.
An infinitesimal motion is called {\em trivial} if $\bmm(v)=\bmm(u)$ for all $u,v\in V$.
It is easy to see that the collection of trivial motions forms a $D$-dimensional vector space.
A body-bar framework is called {\em infinitesimally rigid} if every infinitesimal motion is trivial.

\subsubsection{Body-bar matroids}
The {\em body-bar  matroid} ${\cal B}(G,\bq)$ is defined as a matroid on $E$
whose rank is the maximum size of independent linear equations in (\ref{eq:bar_constraint}) (for unknown $\bmm$).
Namely,  ${\cal B}(G,\bq)$ is a linear matroid on $E$ in which each edge $e=uv$ is represented 
by the following vector in $V_V(=\mathbb{R}^{D|V|})$:
 \begin{equation}
\label{eq:e_vector2}
(\Bvector{}{0}, \Bvector{\cdots}{\cdots}, \Bvector{}{{0}}, 
\Bvector{u}{\overbrace{q_{e}^1, \dots, q_{e}^D}}, \Bvector{}{0}, \Bvector{\cdots}{\cdots}, \Bvector{}{0}, 
\Bvector{v}{\overbrace{-q_{e}^1, \dots, -q_{e}^D}}, \Bvector{}{0}, \Bvector{\cdots}{\cdots}, \Bvector{}{0}).
\end{equation}
Notice that, unlike the union of $D$ copies of the graphic matroid, $[q_e^1,\dots, q_e^D]$ is restricted to $Gr(2,W)$ for each $e\in E$ (compare to (\ref{eq:e_vector})).
The direct product of this restricted space over all edges is called the {\em bar-configuration space} ${\cal C}$.

A bar-configuration $\bq$ is called {\em generic} if the rank of every $F\subseteq E$ in ${\cal B}(G,\bq)$ is maximized among all bar-configurations.
As pointed in \cite{whiteley:88}, it can be seen that almost all bar-configurations $\bq$ are generic as follows. 
Let $B(\bq)$ be the $|E|\times D|V|$-matrix representing ${\cal B}(G,\bq)$.
%, i.e.,
%each row is corresponding to an edge and has form (\ref{eq:e_vector2}).
Note that the rank of ${\cal B}(G,\bq)$ decreases only if a minor of $B(\bq)$ vanishes.
Each minor of $B(\bq)$ defines an algebraic variety $S$ of ${\cal C}$,
which is lower-dimensional than ${\cal C}$ since a polynomial generating $S$ is linear with respect to $q_e^1,\dots, q_e^D$ for each $e\in E$.
Thus, ${\cal C}\setminus S$ is  a dense subset of ${\cal C}$.
Since there are a finite number of minors in $B(\bq)$, the set of points in ${\cal C}$ in which no minor vanishes is also a dense subset of ${\cal C}$.
In other words, almost all bar-configurations are generic.

Notice that, once we assume generic bar-configurations, the rank of ${\cal B}(G,\bq)$ is determined only by $G$.
We hence define the {\em generic body-bar matroid} ${\cal B}(G)$ as ${\cal B}(G,\bq)$ with a (any) generic bar-configuration $\bq$.   
The following result is proved by Tay \cite{tay:84}.
Simpler proofs based on tree-decompositions are given in \cite{white:whiteley:87,whiteley:88}.
We shall provide a proof from our viewpoint.

\begin{theorem}[Tay~\cite{tay:84}]
\label{theorem:body_bar}
Let $G=(V,E)$ be a graph.
Then, ${\cal B}(G)=D{\cal G}(G)$.
\end{theorem}
\begin{proof}
From the discussion given in Subsection~\ref{subsec:graphic},
$D{\cal G}(G)$ is equal to the generic matroid ${\cal M}({\cal A})$ 
associated with the flat family ${\cal A}=\{A_e : e\in E\}$ defined by 
\begin{equation}
\label{eq:e_space2}
A_e=\{ 
[\Bvector{}{0}, \Bvector{\cdots}{\cdots}, \Bvector{}{0}, \Bvector{u}{\bal}, \Bvector{}{0}, \Bvector{\cdots}{\cdots} , \Bvector{}{0},
\Bvector{v}{-\bal}, \Bvector{}{0}, \Bvector{\cdots}{\cdots}, \Bvector{}{0}]
 : 
[\bal]\in \mathbb{P}^{D-1} \}\subseteq \mathbb{P}(V_V).
\end{equation}

In order to prove ${\cal B}(G)={\cal M}({\cal A})$,
it is sufficient to show that the representative point $x_e$ of $A_e$ 
(that defines ${\cal M}({\cal A})$) can be taken to be in general position from 
\begin{equation}
\hat{A}_e=\{[\Bvector{}{0}, \Bvector{\cdots}{\cdots}, \Bvector{}{0},\Bvector{u}{\bal}, \Bvector{}{0}, \Bvector{\cdots}{\cdots},\Bvector{}{0},
\Bvector{v}{-\bal}, \Bvector{}{0}, \Bvector{\cdots}{\cdots}, \Bvector{}{0}]  :  [\bal]\in Gr(2,W)\}\subseteq \mathbb{P}(V_V).
\end{equation}
%Note that $\hat{A}_e$ spans $A_e$.
Specifically, we need to show that there exists $X=\{x_e\in \hat{A}_e :  e\in E\}$ such that,
for each $X'\subseteq X$ and $x_e\in X'$,
\[
 x_e \in \overline{X'-x_e} \Rightarrow A_e \subseteq \overline{X'-x_e},
\]
(c.f.~(\ref{eq:point_generic})).
Let us consider the case $d=3$ (and $D=6$). 
%For $e=uv\in E$, consider a bijection $f_{e}:A_e\rightarrow \mathbb{P}(\bigwedge^2 W)$ defined by:
%\[ 
%[\Bvector{}{0}, \Bvector{}{\cdots}, \Bvector{}{0}, \Bvector{u}{p_1,\dots,p_6}, \Bvector{}{0}, \Bvector{}{\cdots}, \Bvector{}{0},
%\Bvector{v}{-p_1,\dots, -p_6}, \Bvector{}{0}, \Bvector{}{\cdots}, \Bvector{}{0}]
%\quad \longmapsto \quad [p_1,\dots, p_6] \in \mathbb{P}^{5}.
%\
For $e=uv\in E$, let us pick a point
\begin{equation*}
x_e=
[\Bvector{}{0}, \Bvector{\cdots}{\cdots}, \Bvector{}{0}, \Bvector{u}{\overbrace{x_e^1,\dots, x_e^6}}, \Bvector{}{0}, \Bvector{\cdots}{\cdots} , \Bvector{}{0},
\Bvector{v}{\overbrace{-x_e^1,\dots, -x_e^6}}, \Bvector{}{0}, \Bvector{\cdots}{\cdots}, \Bvector{}{0}]\in A_e.
\end{equation*}
Then, $x_e\in \hat{A}_e$ if and only if $x_e^1x_e^6-x_e^2x_e^5+x_e^3x_e^4=0$. 
We now focus on a $5$-dimensional affine space $\mathbb{A}$ by setting $x_e^4=1$.
Note that $Gr(2,W)\cap \mathbb{A}$ is a smooth $4$-dimensional manifold parameterized by $x_e^1, x_e^2, x_e^5, x_e^6$
since $x_e^3=-x_e^1x_e^6+x_e^2x_e^5$.

Let us take $x_e$ 
so that the set of parameters $x_e^1, x_e^2, x_e^5, x_e^6$ for all $e\in E$ is algebraically independent over $\mathbb{Q}$.
Suppose, for a contradiction, that  $x_e\in \overline{X'-x_e}$ but $A_e\not\subset \overline{X'-x_e}$ for some $e=uv$.
%Let $L=\overline{X'-x_e}\cap A_e$.
%Then, since $A_e\not\subset \overline{X'-x_e}$ but $x_e\in \overline{X'-x_e}$, 
%$L$ is a nonempty linear subvariety of $A_e$.
%
Let us consider a hyperplane $H$ of $\mathbb{P}(V_V)$ that contains $\overline{X'-x_e}$ but does not contain $A_e$.
%Let us express such a hyperplane $H$ by $\{[p^0,\dots, p^{6}]\in \mathbb{P}^{5}: \sum_{i=0}^{D-1} a^ip^i=0\}$.
We can take such a hyperplane $H$ so that each coefficient is written as a polynomial of 
$\{x_{e'}^1, x_{e'}^2, x_{e'}^5, x_{e'}^6: e'\in E-e\}$ over $\mathbb{Q}$.
Moreover, $H\cap \hat{A}_e$ is a lower-dimensional subspace of $\hat{A}_e$ since $Gr(2,W)$ is quadratic and irreducible.
In particular, $H$ does not contain $\hat{A}_e$.
%This means that $\sum_{i=0}^{D-1} a^ip^i=0$ is not identically zero over $[p^1,\dots, p^D]\in Gr(2,W)$.
Therefore, if $x_e\in H$, then $\{x_e^1,x_e^2,x_e^5,x_e^6 :  e\in E\}$ satisfies a nontrivial algebraic relation over $\mathbb{Q}$, contradicting the choice of $x_e$.

The general $d$-dimensional case follows in the same way based on  the following fact.
If $Gr(2,W)$ is restricted to a $(D-1)$-dimensional affine space $\mathbb{A}$ by fixing one coordinate,
then $Gr(2,W)\cap \mathbb{A}$ is known to be a smooth $2(d-1)$-dimensional manifold (see, e.g.,\cite{Hassett200705}).
Moreover, each coordinate of a point in $Gr(2,W)\cap \mathbb{A}$ is written as 
a rational function of $2(d-1)$ parameters with coefficients in $\mathbb{Q}$.
Thus, we can apply the exactly same argument.   
\end{proof}

\section{Body-rod-bar Frameworks}
\label{sec:rod_bar}
We now provide our main result on the generic rigidity of body-rod-bar frameworks.
We first introduce a counting matroid defined on graphs in Subsection~\ref{sec:rod_combinatorial},
and then in Subsection~\ref{sec:rod_algebraic} 
we show that generic rigidity of body-rod-bar frameworks can be characterized by the combinatorial matroid.

\subsection{Combinatorial truncated matroids}
\label{sec:rod_combinatorial}
\subsubsection{Count matroids}
Let $G=(V,E)$ be a graph with an (ordered) partition ${\cal P}=\{B,R\}$ of $V$ into two subsets 
(where $B$ and $R$ will represent a set of bodies and a set of rods, respectively, in the next subsection).
We define an integer-valued function $f$ on $E$ defined by
\begin{equation}
\label{eq:truncated_function}
f(F)=D(|V(F)|-1)-|R(F)| \qquad (F\subseteq E),
\end{equation}
where $R(F)$ denotes the set of vertices in $R$ spanned by $F$, and $D={d+1\choose 2}$ as in Section~\ref{sec:body_bar}. 
Then, $f$ is a monotone submodular function on $E$, 
since $f(F)=D|B(F)|+(D-1)|R(F)|-D$ and 
$|B(\cdot)|$ and $|R(\cdot)|$ are both monotone and submodular. 
Thus, $f$ induces the matroid $(E,r_f)$ on $E$, denoted ${\cal M}_f(G,{\cal P})$.
%(Namely, $F\subseteq E$ is independent if and only if $|F'|\leq f(F')$ for any nonempty $F'\subseteq F$.
If the bipartition ${\cal P}$ is clear from the context, we abbreviate it and simply denote ${\cal M}_f(G)$.
This matroid is a special case of so-called {\em count matroids} on undirected graphs, 
see e.g., \cite[Section 13.5]{Frank2011} for more detail.

We denote by $f\circ G$ the graph obtained from $G$ by replacing each edge $e$ by $f(e)$ parallel copies of $e$ (see Figure~\ref{fig:fG}).
Also, $f\circ e$ denotes the {\em set} of corresponding copies of $e$, and let $f\circ F=\bigcup_{e\in F}f\circ e$. 
We can naturally extend $f$ to that on $f\circ E$ by setting $f(F)=D|V(F)|-D-|R(F)|$ for $F\subseteq f\circ E$.

\begin{figure}[t]
\centering
\begin{minipage}{0.3\textwidth}
\centering
\includegraphics[scale=1]{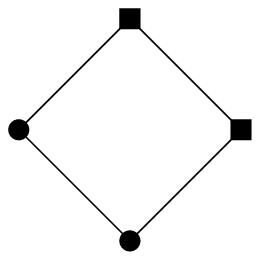}
\par
$G$
\end{minipage}
\begin{minipage}{0.3\textwidth}
\centering
\includegraphics[scale=1]{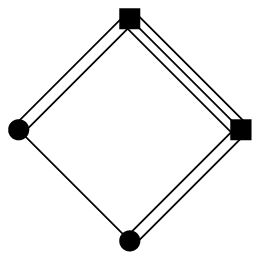}
\par
$f\circ G$
\end{minipage}
\caption{Example of $f\circ G$ for $D=3$, where circles and squares represent vertices of $R$ and $B$, respectively.}
\label{fig:fG}
\end{figure}

Let us consider $\hat{f}:2^E\rightarrow \mathbb{Z}$ defined by (\ref{eq:g_hat}), i.e., for $F\subseteq E$,
\begin{equation}
\label{eq:f_hat}
\hat{f}(F)=\min\{\mbox{$\sum_{i=1}^k$}(D(|V(F_i)|-1)-|R(F_i)|) : \text{a partition $\{F_1,\dots, F_k\}$ of $F$} \}. 
\end{equation}
As mentioned in Subsection~\ref{subsec:polymatroids}, $\hat{f}$ is a monotone submodular function satisfying $f(\emptyset)=0$,
and thus $(E,\hat{f})$ forms a polymatroid, denoted by ${\cal PM}_f(G,{\cal P})$ (or simply by ${\cal PM}_f(G)$).
The following lemma implies that ${\cal PM}_f(G)$ is essentially the same as ${\cal M}_f(f\circ G)$.
\begin{lemma}
\label{lemma:circ}
For any $F\subseteq E$, $\hat{f}(F)=r_f(f\circ F)$.
Namely, the rank of $F\subseteq E$ in ${\cal PM}_f(G)$ is equal to the rank of $f\circ F$ in ${\cal M}_f(f\circ G)$.
\end{lemma}
\begin{proof}
Recall that, for any $F\subseteq E$, $r_f(f\circ F)$ is written as
$r_f(f\circ F)=\min\{|F_0|+\sum_{i=1}^kf(F_i)\}$, where the minimum is taken over partitions $\{F_0,F_1,\dots, F_k\}$ of $f\circ F$
such that $F_1,\dots, F_k\neq \emptyset$ (see (\ref{eq:matroid_rank})).
Let $\{F_0^*,F_1^*,\dots, F_k^*\}$ be a partition of $f\circ F$ that attains that minimum.
Since $|f\circ e|=f(e)$ for every $e\in E$, we may assume $F_0^*=\emptyset$.
Also, since $f(f\circ F)=f(F)$ for any $F\subseteq E$, we may assume that
each $F_i^*(\subseteq f\circ F)$ is written as $F_i^*=f\circ F_i'$ for some $F_i'\subseteq F$.
Thus, $r_f(f\circ F)$ is actually written as $r_f(f\circ F)=\min\{\sum_{i=1}^kf(f\circ F_i')\}=\min\{\sum_{i=1}^kf(F_i')\}$,
where the minimum is taken over all partitions $\{F_1',\dots, F_k'\}$ of $F$.
This is exactly the definition of $\hat{f}(F)$. 
\end{proof}
A reduction technique of general polymatroids to matroids can be found in, e.g.,\cite[Section~44.6b]{Schriver}.

\subsubsection{Properties of ${\cal M}_f$}
We now show several properties of ${\cal M}_f(G,{\cal P})$ for a graph $G=(V,E)$ with a bipartition ${\cal P}=\{B,R\}$ of $V$.
(These lemmas are generally known for count matroids. We provide proofs for the completeness.)
\begin{lemma}
\label{lemma:f_circuit}
Let $C$ be a circuit of ${\cal M}_f(G)$.
Then, $r_f(C)=f(C)$.
\end{lemma}
\begin{proof}
Since $C$ is a minimal dependent set, $|C|>f(C)$ and $|C|-1=|C-e|\leq f(C-e)\leq f(C)$ for any $e\in C$.
This implies $|C|=f(C)+1$.
Thus, $r_f(C)=|C|-1=f(C)$.
\end{proof}
\begin{lemma}
\label{lemma:f_connected}
Let $F\subseteq E$ be a nontrivial $M$-connected set in ${\cal M}_f(G)$.
Then, $r_f(F)=f(F)$. 
\end{lemma}
\begin{proof}
Suppose $r_f(F)<f(F)$. 
Then, there are $u,v\in V(F)$ with $uv\notin F$ such that $r_f(F+uv)=r_f(F)+1$.
Let us take two distinct edges $e$ and $e'$ of $F$ incident to $u$ and $v$, respectively.
(It is easy to see that such two edges exist since $F$ is $M$-connected.)
Since $F$ is $M$-connected, there is a circuit $C\subseteq F$ that contains $e$ and $e'$.
Then, by Lemma~\ref{lemma:f_circuit} and by $f(C+uv)=f(C)$,  
we obtain $r_f(C+uv)\leq f(C+uv)=f(C)=r_f(C)$, implying $r_f(C+uv)=r_f(C)$.
In other words, $uv$ is contained in the closure of $C$. This contradicts $r_f(F+uv)=r_f(F)+1$.
\end{proof}
\begin{lemma}
\label{lemma:M_closure}
Let $F\subseteq E$ be a nontrivial $M$-connected set in ${\cal M}_f(G)$.
Then, the closure  of $F$, that is, $\{e\in E(G)\colon r_f(F+e)=r_f(F)\}$, 
is the set of edges induced by $V(F)$.
In particular, if $F$ is an $M$-connected component, then $(V(F),F)$ is an induced subgraph.
\end{lemma}
\begin{proof}
Since $f(F+e)=f(F)$ holds for any edge $e$ induced by $V(F)$, 
the claim follows from Lemma~\ref{lemma:f_connected}.
\end{proof}
\subsubsection{Properties of ${\cal PM}_f$}
Let us consider ${\cal M}_f(f\circ G)$ for a graph $G=(V,E)$ with a bipartition ${\cal P}$.
By Lemma~\ref{lemma:M_closure}, an $M$-connected component $C$ of ${\cal M}_f(f\circ G)$ is 
either trivial or of the form $C=f\circ F$ for some $F\subseteq E$ with $|F|\geq 2$.
The $M$-connected component decomposition of ${\cal M}_f(f\circ G)$ 
thus induces a unique partition $\{C_1,\dots, C_k\}$ of $E$
such that $C_i$ is singleton or $f\circ C_i$ is an $M$-connected component in ${\cal M}_f(f\circ G)$.
The following lemma says that this partition coincides with the $P$-connected component decomposition of ${\cal PM}_f(G)$. 

\begin{lemma}
\label{lemma:minimizer0}
For a graph $G=(V,E)$ with a bipartition ${\cal P}=\{B,R\}$ of $V$, the following holds:
\begin{description}
\item[(i)] 
Any nontrivial $M$-connected component $X$ of ${\cal M}_f(f\circ G)$ 
can be written as $X=f\circ F$ for some nontrivial $P$-connected component $F\subseteq E$.  
\item[(ii)] If $F\subseteq E$ is a nontrivial $P$-connected set in ${\cal PM}_f(G)$,
then  $f\circ F$ is $M$-connected in ${\cal M}_f(G)$.
\item[(iii)] The $P$-connected component decomposition $\{C_1,\dots, C_k\}$ of ${\cal PM}_f(G)$ is a minimizer of the right hand side of (\ref{eq:f_hat}).
\end{description}
\end{lemma}
\begin{proof}
(i) and (ii) are direct consequences of Lemma~\ref{lemma:circ}.
%Indeed, as we have seen in the proof of Lemma~\ref{lemma:circ}, $r$
%By Lemma~\ref{lemma:M_closure}, 
%any nontrivial $M$-connected component $X$ 
%can be written as $X=f\circ F$ for some $F\subseteq E$.
%$F$ must be a $P$-connected component by Lemma~\ref{lemma:circ}. Thus we have (i).
%
%If $f\circ F$ is not $M$-connected, then 
%$f\circ F$ can be decomposed into at least two 
%$M$-connected components of ${\cal M}_f(f\circ G)$ restricted to $f\circ F$.
%By (i) and $|e\circ f|=f(e)$, such a decomposition induces a decomposition of $F$ into 
%at least two $P$-connected components. This implies (ii).

For the last claim, 
Lemma~\ref{lemma:circ}, Lemma~\ref{lemma:M_closure} and (ii) imply 
$\hat{f}(E)=\sum_{i=1}^k \hat{f}(C_i)=\sum_{i=1}^k r_f(f\circ C_i)=\sum_{i=1}^kf(f\circ C_i)=\sum_{i=1}^k f(C_i)$.
\end{proof}

For a simple graph $G=(V,E)$, 
it is sometimes useful to introduce the underlying complete simple graph  $K(V)$ on $V$ that contains $G$, 
and extend ${\cal PM}_f(G,{\cal P})$ to ${\cal PM}_f(K(V),{\cal P})$.
We shall denote by ${\rm cl}$ the closure operator of 
${\cal PM}_f(K(V), {\cal P})$, i.e., ${\rm cl}(F)=\{uv\in K(V):\hat{f}(F+uv)=\hat{f}(F)\}$ for $F\subseteq E$.
Then, by Lemma~\ref{lemma:minimizer0}, ${\rm cl}(F)$ forms the complete graph on $V(F)$ if $F$ is $P$-connected. 

The following lemmas are key observations used in the proof of main theorem (Theorem~\ref{theorem:main2}).
\begin{lemma}
\label{lemma:degree}
Let $G=(V,E)$ be a connected simple graph with a bipartition ${\cal P}=\{B,R\}$ of $V$. 
Suppose $D\geq 6$. 
Then $G$ has (i) three vertices each of which is spanned by exactly two $P$-connected components of ${\cal PM}_f(G)$ 
or (ii) a vertex that is spanned by only one $P$-connected component.
\end{lemma}
\begin{proof} 
Let $\{C_1,\dots, C_k\}$ be the $P$-connected component decomposition of ${\cal PM}_f(G)$.
%If $C_i$ contains edges $e_2, \dots, e_j$ parallel to $e_1\in C_i$, then we remove $e_2,\dots, e_j$ from $G$.
%Since the closure of $f\circ e_1$ in ${\cal M}_f(f\circ G)$ contains all elements of $f\circ e_2,\dots, f\circ e_j$,
%$\{C_1,\dots, C_i\setminus\{e_2,\dots, e_j\}, C_{i+1},\dots, C_k\}$ is the $P$-connected component decomposition of the resulting polymatroid.
%Thus, it is sufficient to show the claim for the resulting polymatroid.
%Therefore, we may assume that $G$ has no parallel edges in the subsequent discussion.
Note then, since $G$ is simple, any nontrivial $P$-connected component $C_i$ satisfies $|V(C_i)|\geq 3$.

For each nontrivial $C_i$, we consider the following graph operation on $G$, called the {\em simplification} of $C_i$;
remove $C_i$, insert a new vertex $v_c$ to $B$,
and connect each vertex of $V(C_i)$ with $v_c$.
Namely, we replace the induced subgraph $(V(C_i),C_i)$ by the star $(V(C_i)\cup\{v_c\}, S)$ with the centered new vertex $v_c$ and the set $S$ of edges between $v_c$ and $V(C_i)$ (see Figure~\ref{fig:simplification}).
\begin{figure}[th]
\centering
\includegraphics[scale=0.8]{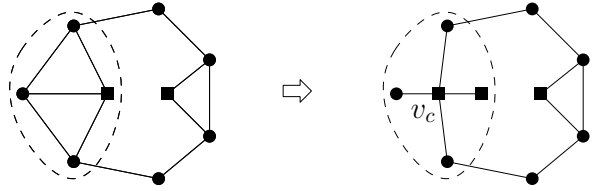}
\caption{Simplification.}
\label{fig:simplification}
\end{figure}

\begin{claim}
\label{claim:deg}
Let $C$ be a nontrivial $P$-connected component of ${\cal PM}_f(G,{\cal P})$.
Let $G'$  be the graph obtained by the simplification of $C$, where we denote $V(G')=V\cup\{v_c\}$ and $E(G')=(E\setminus C)\cup S$,
with the bipartition ${\cal P}'=\{B\cup\{v_c\},R\}$ of $V(G')$.
Then, 
%\begin{itemize}
%\item[(a)] $\widetilde{S}$ is independent in ${\cal M}_f(\widetilde{G'},{\cal P}')$ with $r_f(\widetilde{S})=f(\widetilde{S})$;
each new edge $e\in f\circ S$ is a coloop in ${\cal M}_f(f\circ G',{\cal P}')$.
%\item[(c)] $r_f((\widetilde{E}\setminus \widetilde{C})\setminus\widetilde{S})=r_f(\widetilde{E})-D$.
%\end{itemize}
\end{claim}
\begin{proof}
%(a) This can be easily checked from the definition of $f$ and $\widetilde{S}$.
From the definition of $f$, it is easy to check  that $f\circ S$ is independent in ${\cal M}_f(f\circ G',{\cal P}')$.
Since  $C$ is a $P$-connected component,
we have ${\rm cl}(S)\cap {\rm cl}(C')=\emptyset$ 
for any other $P$-connected component $C'$ of ${\cal PM}_f(G,{\cal P})$.
This implies that there is no circuit of ${\cal M}_f(f\circ G',{\cal P}')$ 
that intersects both $f\circ(E\setminus C)$ and $f\circ S$.
Since $f\circ S$ is independent, there is also no circuit within $f\circ S$ and thus no circuit that contains $e\in f\circ S$ in ${\cal M}_f(f\circ G',{\cal P}')$.
\end{proof}

Claim~\ref{claim:deg} implies that, if we apply the simplification of the $P$-connected component $C_i$,
then no new nontrivial $P$-connected component appears, and $C_1,\dots, C_{i-1}, C_{i+1},\dots, C_k$ are all nontrivial $P$-connected components in the resulting polymatroid.  
Hence, we may apply the simplifications for all $C_1,\dots, C_k$ simultaneously.
Let $G''$ be the resulting graph with the corresponding bipartition ${\cal P}''$ of $V(G'')$ after the simplifications. 
Notice that the degree of each vertex $v\in V(G)$ in $G''$ corresponds to the number of $P$-connected components among $C_1,\dots, C_k$ that span $v$ in $G$. 
We also remark that each vertex of $V(G'')\setminus V(G)$ has degree at least three since $|V(C_i)|\geq 3$.
Thus, to complete the proof, it is sufficient to show that 
$G''$ has at least three vertices of degree $2$ or a vertex of degree $1$.
To see this, observe that $f\circ E(G'')$ is independent in ${\cal M}_f(f\circ G'',{\cal P}'')$ by Claim~\ref{claim:deg}.
So, we have $|f\circ E(G'')|=r_f(f\circ E(G''))$. 
This implies 
$(D-2)|E(G'')|\leq \sum_{e\in E(G'')}f(e)=|f\circ E(G'')|=r_f(f\circ E(G''))\leq D|V(G'')|-D$.
Let $d_{\rm avg}$ be the average degree of $G''$.
Then, we have
\begin{equation*}
\label{eq:deg1} 
\mbox{$d_{\rm avg}=\frac{2|E(G'')|}{|V(G'')|}\leq \frac{2D}{D-2}\left(1-\frac{1}{|V(G'')|}\right)$}.
\end{equation*}
Suppose there is no vertex of degree $1$.
Denoting the set of vertices of degree $2$ in $G''$ by $V_2$,
we have 
\begin{equation*}
\label{eq:deg2}
\mbox{$d_{\rm avg}\geq 3-\frac{|V_2|}{|V(G'')|}$}.
\end{equation*}
Putting them together, we obtain 
\begin{equation*}
\mbox{$|V_2|\geq \frac{2D}{D-2}+\frac{D-6}{D-2}|V(G'')|\geq \frac{2D}{D-2}+\frac{D-6}{D-2}=3$}.
\end{equation*}
(where we used $D\geq 6$ and $|V(G'')|\geq 1$).
This completes the proof.
\end{proof}

{\it Remark.} Lemma~\ref{lemma:degree} does not hold for $d=2$ and $D=3$.
For example, in the cube graph, all $P$-connected components are trivial 
and hence each vertex is spanned by three $P$-connected components since each vertex has degree $3$.   
\qed

\begin{lemma}
\label{lemma:stronger}
Let $G=(V,E)$ be a  simple graph for which $E$ is $P$-connected in ${\cal PM}_f(G,{\cal P})$.
Suppose further that  there are two disjoint nonempty $P$-connected sets $C_1$ and $C_2$ both of which span a vertex $u\in V$.
Then, $G$ contains a $P$-connected set $C$ 
such that $C_1\subseteq C \subseteq E\setminus C_2$ and  
$uv\in {\rm cl}(C)\cap {\rm cl}(E\setminus C)$ for some $uv\in K(V)$. 
\end{lemma}
\begin{proof}
Let us take an inclusion-wise maximal $P$-connected set $C$ such that $C_1\subseteq C\subseteq E\setminus C_2$.
Since $E$ is $P$-connected, 
we have ${\rm cl}(C)\cap {\rm cl}(E\setminus C)\neq \emptyset$, and hence there is an edge $vw\in K(V)$ such that $vw\in {\rm cl}(C)\cap {\rm cl}(E\setminus C)$.
If either $v=u$ or $w=u$,  then $C$ satisfies the required property.
Thus, suppose contrary that every edge in  ${\rm cl}(C)\cap {\rm cl}(E\setminus C)$ is not incident to $u$.
Let $C'$ be a $P$-connected set in $E\setminus C$ with $vw\in {\rm cl}(C')$.
Since $vw\in {\rm cl}(E\setminus C)$, such $C'$ exist. ($C'=\{vw\}$ may hold if $vw\in E\setminus C$.)

If $C_2\cap C'\neq \emptyset$, then $C_2\cup C'$ is $P$-connected, and hence ${\rm cl}(C_2\cup C')$ forms the complete graph on $V(C_2\cup C')$. 
Since  $u\in V(C_2)$ and $v\in V(C')$, we obtain  $uv\in {\rm cl}(C_2\cup C')\subseteq {\rm cl}(E\setminus C)$.
On the other hand, since $C$ is $P$-connected with $u,v\in V(C)$, we also have $uv\in {\rm cl}(C)$.
These however contradicts 
that every edge in  ${\rm cl}(C)\cap {\rm cl}(E\setminus C)$ is not incident to $u$

If $C_2\cap C'=\emptyset$, then $C\cup C'$ is $P$-connected 
since ${\rm cl}(C)\cap {\rm cl}(C')$ is nonempty,
and thus $C\cup C'$ is $P$-connected with  $C_1\subseteq C\cup C'\subseteq E\setminus C_2$ and is larger than $C$,
contradicting the choice of $C$. 
\end{proof}

\subsection{Generic body-rod-bar matroids}
\label{sec:rod_algebraic}
\subsubsection{Body-rod-bar frameworks}
A {\em body-rod-bar framework}  is a body-bar framework in which some of bodies are degenerate as $1$-dimensional bodies in the case of $d=3$.
In general dimensional case, a body-rod-bar framework can be defined as a structure consisting of $d$-dimensional subspaces (bodies) 
and $(d-2)$-dimensional flats (rods) mutually linked by $1$-dimensional lines (bars).
(The name of ``rod'' is actually appropriate only for $d=3$.)
We thus define a body-rod-bar framework as $(G,\bq,\br)$, where
\begin{itemize}
\item $G=(V,E)$ is a graph with a bipartition ${\cal P}=\{B,R\}$ of $V$;
\item $\br$ is a mapping called a {\em rod-configuration}:
\begin{align*}
\br:R&\rightarrow Gr(d-1,W) \subseteq \mbox{$\mathbb{P}(\bigwedge^{d-1} W)$} \\
v&\mapsto [\br_v]=[r_v^1,\dots, r_v^D]
\end{align*}
\item $\bq$ is a {\em bar-configuration}:
\begin{align*}
\bq:R&\rightarrow Gr(2,W) \subseteq \mbox{$\mathbb{P}(\bigwedge^{2} W)$} \\
e&\mapsto [\bq_e]=[q_e^1,\dots,q_e^D]
\end{align*}
satisfying  the {\em incidence condition}:
\begin{equation}
\label{eq:incidence}
\bq_e\cdot \br_v=0 \qquad \text{if $e\in E$ is incident to $v\in R$}.
\end{equation}
\end{itemize}
Namely, $\br(v)$ represents a rod associated with $v\in R$,
and $[\br_v]$ denotes the Pl{\"u}cker coordinate of the rod.
Recall that, for $[\bq] \in Gr(2,W)$ and $[\br]\in Gr(d-1,W)$, 
$\bq \cdot \br=0$ holds if and only if the corresponding linear subspaces have nonzero intersection
(equivalently, the corresponding flats have a nonempty intersection). 
Thus, the system (\ref{eq:incidence}) describes incidence constraints between rods and bars. 
Throughout the subsequent discussions, we also impose an additional condition that all rods are distinct, i.e.,  $\br(u)\neq \br(v)$ for any $u,v\in R$ with $u\neq v$. 

As in the case of body-bar frameworks, an infinitesimal motion of $(G,\bq,\br)$ is defined as $\bmm:V\rightarrow \bigwedge^{d-1} W$ satisfying (\ref{eq:bar_constraint}),
and $\bmm$ is called {\em trivial} if $\bmm(u)=\bmm(v)$ for all $u,v\in V$.

For each $v\in R$, define $\bmm_v:V\rightarrow \bigwedge^{d-1} W$ by $\bmm_v(v)=\br_v$ and $\bmm_v(u)=0$ for $u\in V\setminus\{v\}$.
Then, by incidence condition (\ref{eq:incidence}), $\bmm_v$ always satisfies (\ref{eq:bar_constraint}),
and $\bmm_v$ is an infinitesimal motion of $(G,\bq,\br)$.
Conventionally, we also include $\bmm_v$ in the set of trivial motions.
The set of all trivial motions thus forms a $(D+|R|)$-dimensional vector space.
If every motion of $(G,\bq,\br)$ is trivial, it is said to be {\em infinitesimally rigid}.

\subsubsection{Body-rod-bar matroids}
As defined in the body-bar matroid, the {\em body-rod-bar} matroid ${\cal BR}(G,\bq,\br)$ is defined as that on $E$ whose rank is the maximum size of 
independent linear equations in (\ref{eq:bar_constraint}) (for unknown $\bmm$).
From the definition, $(G,\bq,\br)$ is infinitesimally rigid if and only if the rank of ${\cal BR}(G,\bq,\br)$ is $D|V|-(D+|R|)$.
The following theorem is our main result.
\begin{theorem}
\label{theorem:main}
Let $G=(V,E)$ be a graph with a bipartition ${\cal P}=\{B,R\}$ of $V$ 
and $f$ be the function defined by (\ref{eq:truncated_function}).
Suppose $d\geq 3$.
Then, for almost all bar-configurations $\bq$ and almost all rod-configurations $\br$, ${\cal BR}(G,\bq,\br)={\cal M}_f(G,{\cal P})$.
Namely, $I\subseteq E$ is independent in ${\cal BR}(G,\bq,\br)$ if and only if
$|F|\leq D|V(F)|-D-|R(F)|$ for any nonempty $F\subseteq I$.
\end{theorem}

We need to introduce a notation for the proof.
Let $\br:R\rightarrow Gr(d-1,W)$ be a rod-configuration.
For each $v\in R$, let $H_{\br}(v)$ be the dual hyperplane to the point $[\br_v]$ in $\mathbb{P}(\bigwedge^2 W)$,
i.e., $H_{\br}(v)=\{[\bp]\in \mathbb{P}(\bigwedge^2 W) :  \bp\cdot \br_v=0\}$.
For easiness of the description, we also define $H_{\br}(v)$ for $v\in B$ to be $H_{\br}(v)=\mathbb{P}(\bigwedge^2 W)$.
Notice that, due to the incidence condition (\ref{eq:incidence}), the space of $\bq_{uv}$  is restricted to $Gr(2,W)\cap H_{\br}(u)\cap H_{\br}(v)$ for $uv\in E$. 
We hence define two subspaces associated with $e=uv\in E$ as follows: 
\begin{equation}
\label{eq:e_space3}
A_e(\br)=\{ 
[\Bvector{}{0}, \Bvector{}{\cdots}, \Bvector{}{0}, 
\Bvector{u}{\bal}, \Bvector{}{0}, \Bvector{}{\cdots}, \Bvector{}{0},
\Bvector{v}{-\bal}, \Bvector{}{0}, \Bvector{}{\cdots}, \Bvector{}{0}]
 : 
[\bal] \in \mathbb{P}(\mbox{$\bigwedge^2$} W)\cap H_{\br}(u) \cap H_{\br}(v)\},
\end{equation}
\begin{equation}
\label{eq:e_space4}
\hat{A}_e(\br)=\{
[\Bvector{}{0}, \Bvector{}{\cdots}, \Bvector{}{0}, 
\Bvector{u}{\bal}, \Bvector{}{0}, \Bvector{}{\cdots}, \Bvector{}{0},
\Bvector{v}{-\bal}, \Bvector{}{0}, \Bvector{}{\cdots}, \Bvector{}{0}]
 :  [\bal]\in Gr(2,W)\cap H_{\br}(u)\cap H_{\br}(v) \}.
\end{equation}
Also, let ${\cal A}(\br)=\{A_e(\br) : e\in E\}$, and as before let ${\cal A}_F(\br)=\{A_e(\br) : e\in F\}$ for $F\subseteq E$.

The proof of Theorem~\ref{theorem:main} proceeds as follows:
we first show that 
${\cal BR}(G,\bq,\br)$ is equal to the linear matroid ${\cal M}({\cal A}(\br))$ 
associated with the flat family ${\cal A}(\br)$ for almost all configurations (Theorem~\ref{theorem:main1}).
%This can be proved in the same way as Theorem~\ref{theorem:body_bar}. 
We then provide an explicit formula of  the rank of ${\cal A}(\br)$ in terms of the underlying graph $G$ (Theorem~\ref{theorem:main2})
and finally show that  ${\cal M}({\cal A}(\br))$  is indeed equal to ${\cal M}_f(G,{\cal P})$ 
(Corollary~\ref{theorem:main3}).

\begin{theorem}
\label{theorem:main1}
Let $G=(V,E)$ be a graph with a bipartition ${\cal P}=\{B,R\}$.
Then, for almost all rod-configurations $\br$ and bar-configurations $\bq$,
${\cal BR}(G,\bq,\br)={\cal M}({\cal A}(\br))$.
\end{theorem}
\begin{proof}
The proof is basically the same as that of Theorem~\ref{theorem:body_bar}.
Recall that ${\cal BR}(G,\bq,\br)$ is a linear matroid on $E$ in which each element $e=uv\in E$ is represented by
\begin{equation*}
(\Bvector{}{0}, \Bvector{\cdots}{\cdots}, \Bvector{}{0},
\Bvector{u}{\bq_e}, \Bvector{}{0}, \Bvector{\cdots}{\cdots}, \Bvector{}{0},
\Bvector{v}{-\bq_e}, \Bvector{}{0},\Bvector{\cdots}{\cdots}, \Bvector{}{0}),
\end{equation*}
where $[\bq_e]$ is restricted to $Gr(2,W)\cap H_{\br}(u)\cap H_{\br}(v)$ in the case of body-rod-bar frameworks.
Hence, to prove ${\cal BR}(G,\bq,\br)={\cal M}({\cal A}(\br))$,
it is sufficient to show that a representative point 
$x_{e}=[0,\dots, 0, x_e^1,\dots, x_e^D, 0,\dots, 0, -x_e^1,\dots, -x_e^D,0,\dots, 0]$ of $A_e(\br)$ can be taken from $\hat{A}_e(\br)$
so that $X=\{x_e :  e\in E\}$ is  in generic position (in the sense of definition (\ref{eq:point_generic})). 

Let us consider the case $d=3$.
Let us take $\br$ so that $\br(u)\neq \br(v)$ for each $u,v\in V$ with $u\neq v$.
Then, for each $e=uv\in E$, $A_e(\br)$ is isomorphic to 
$\mathbb{P}(\bigwedge^2 W)\cap H_{\br}(u)\cap H_{\br}(v)=\mathbb{P}^k$,
where $k=3$ if $u,v\in R$; $k=4$ if either $u\in R$ or $v\in R$; otherwise $k=5$.
Recall that the quadratic variety $Gr(2,W)\cap H_{\br}(u)\cap H_{\br}(v)$ is singular 
if the associated matrix is singular.
Since the determinant of the associated matrix is a polynomial of entries of $\br(u)$ and $\br(v)$,
$Gr(2,W)\cap H_{\br}(u)\cap H_{\br}(v)$ becomes a non-singular quadratic variety of $\mathbb{P}^k$ for almost all rod-configurations $\br$.
Then, by setting $x_e^4=1$, it can be easily checked that  $Gr(2,W)\cap H_{\br}(u)\cap H_{\br}(v)$ can be parameterized by $x_e^1$ and $x_e^2$ such that
the rest of coordinates $x_e^3,\dots, x_e^6$ are described as rational functions of $x_e^1$ and $x_e^2$ with coefficients in $\mathbb{Q}$.
If we take $x_e$ so that $\{x_e^1,x_e^2 : e\in E\}$ is algebraically independent over $\mathbb{Q}$,
$X=\{x_e : e\in E\}$ is in generic position by the same reason as  the proof of Theorem~\ref{theorem:body_bar}.

The general $d$-dimensional case follows in the same way, 
as each coordinate of a point in $Gr(2,W)\cap \mathbb{A}$ is written as 
a rational function of $2(d-1)$ parameters among $x^{i,j}_e\ (1\leq i<j\leq d+1)$, 
if $Gr(2,W)$ is restricted to a $(D-1)$-dimensional affine space $\mathbb{A}$ (see, e.g., \cite{Hassett200705}).
\end{proof}

As noted above, ${\cal BR}(G,\bq,\br)$ takes the rank at most $D|V|-D-|R|$ since the corresponding framework $(G,\bq,\br)$ always has $D+|R|$ trivial motions.
The same argument can be applied to show the following fact. 
\begin{lemma}
\label{lemma:easy_direction}
Let $G=(V,E)$ be a graph with a bipartition ${\cal P}=\{B,R\}$ of $V$.
Then, for any rod-configuration $\br$ such that $\br(u)\neq \br(v)$ for $u,v\in R$ with $u\neq v$,
$\rank(\overline{{\cal A}(\br)})\leq D|V|-D-|R|$.
\end{lemma}

The following is a key result for proving Theorem~\ref{theorem:main}.
\begin{theorem}
\label{theorem:main2}
Let $G=(V,E)$ be a graph with a bipartition ${\cal P}=\{B,R\}$ of $V$.
If $d\geq 3$, then for almost all rod-configurations $\br$,
\begin{equation}
\label{eq:main2}
\rank(\overline{{\cal A}(\br)})=\min\{\mbox{$\sum_{i=1}^k$}(D|V(E_i)|-D-|R(E_i)|)\},
\end{equation}
where the minimum is taken over all partitions $\{E_1,\dots, E_k\}$ of $E$ into nonempty subsets.
Namely, the linear polymatroid ${\cal PM}({\cal A}(\br))$ defined by ${\cal A}(\br)$ is equal to the combinatorial polymatroid ${\cal PM}_f(G,{\cal P})$ 
for almost all rod-configurations $\br$.
\end{theorem}
One direction of Theorem~\ref{theorem:main2} is straightforward from Lemma~\ref{lemma:easy_direction};
For any partition $\{E_1,\dots, E_k\}$ of $E$, we have
%\begin{equation}
%\label{eq:easy_direction} 
$\rank(\overline{{\cal A}(\br)})\leq \sum_{i=1}^k \rank(\overline{{\cal A}_{E_i}(\br)})\leq \sum_{i=1}^k(D|V(E_i)|-D-|R(E_i)|)$.
%\end{equation}
Since the proof is not short, the converse direction is left to the next subsection.

\begin{corollary}
\label{theorem:main3}
Let $G=(V,E)$ be a graph with a bipartition ${\cal P}=\{B,R\}$ of $V$.
If $d\geq 3$, then  ${\cal M}({\cal A}(\br))={\cal M}_f(G,{\cal P})$ for almost all rod-configurations $\br$.
\end{corollary}
\begin{proof}
This directly follows from Theorem~\ref{theorem:main2} and general results on polymatroids reviewed in Section~\ref{sec:flats}.
Indeed, by Theorem~\ref{theorem:flat_matroid} and Theorem~\ref{theorem:main2}, the rank of $F\subseteq E$ in ${\cal M}({\cal A}(\br))$ is written as 
\begin{equation*}
%\label{eq:main3_1}
\min\{|F_0|+\mbox{$\sum_{i=1}^k$}(D|V(F_i)|-D-|R(F_i)|)\}
\end{equation*}
where the minimum is taken over all partitions $\{F_0,F_1,\dots, F_k\}$ of $F$
such that $F_1,\dots, F_k\neq \emptyset$.
This is exactly the rank formula (\ref{eq:matroid_rank}) of the matroid induced by $f$.
\end{proof}
Combining Theorem~\ref{theorem:main1} and Corollary~\ref{theorem:main3},
we conclude the proof of Theorem~\ref{theorem:main}.

\medskip

{\em Remark.} Due to the absence of Lemma~\ref{lemma:degree}, the proof of Theorem~\ref{theorem:main2} (given in the next subsection) could not be applied to the 2-dimensional case.
Although Theorem~\ref{theorem:main2} can be proved even for the 2-dimensional case with a slightly different manner, we would not go into the detail as there are already many simpler proofs 
for this case~\cite{lovasz:1982,whiteley:88,tay:whiteley:1985,Whitley:1997}.  \qed

\medskip
Theorem~\ref{theorem:main} is restated in terms of rigidity as follows.
\begin{corollary}
\label{cor:body_rod_bar}
Let $G=(V,E)$ be a graph with a bipartition ${\cal P}=\{B,R\}$ of $V$.
Then, there exists 
a bar-configuration $\bq$ and a rod-configuration $\br$ 
such that the body-rod-bar framework $(G,\bq,\br)$ is minimally infinitesimally rigid 
(i.e., removing any bar results in a flexible framework) in $\mathbb{R}^d$ 
if and only if $G$ satisfies the following counting conditions:
\begin{itemize}
\item $|E|=D|B|+(D-1)|R|-D$;
\item $|F|\leq D|B(F)|+(D-1)|R(F)|-D$ for any nonempty $F\subseteq E$.
\end{itemize}
\end{corollary}

Tay's combinatorial characterization of rod-bar frameworks is an easy consequence.
\begin{corollary}[Tay\cite{tay1991linking,tay:89}]
\label{cor:rod_bar}
Let $G=(V,E)$ be a graph.
Then, there exists 
a bar-configuration $\bq$ and a rod-configuration $\br$ 
such that the rod-bar framework $(G,\bq,\br)$ is minimally infinitesimally rigid in $\mathbb{R}^d$ 
if and only if $G$ satisfies the following counting conditions:
\begin{itemize}
\item $|E|=(D-1)|V|-D$;
\item $|F|\leq (D-1)|V(F)|-D$ for any nonempty $F\subseteq E$.
\end{itemize}
\end{corollary}
\begin{proof}
The rod-bar framework $(G,\bq,\br)$ is a body-rod-bar framework with $R=V$ and $B=\emptyset$.
In this case $D(|V(F)|-1)-|R(F)|=(D-1)|V(F)|-D$ for each $F\subseteq E$.
Therefore, the statement follows from Corollary~\ref{cor:body_rod_bar}.
\end{proof}

\subsection{Proof of Theorem~\ref{theorem:main2}}
\label{subsec:main2}
\begin{proof}
We have already seen ``$\leq$'' direction of (\ref{eq:main2}).
The converse direction is proved  by induction on the lexicographical ordering of the triples $(|V|,|R|,|E|)$.
Since the base case is trivial, let us consider the general case.
Since $A_e(r)=A_{e'}(r)$ for any parallel $e$ and $e'$, we may assume 
that $G$ is simple throughout the proof.

We split the proof into two cases depending on whether $B=\emptyset$ or not.

\subsubsection{Case of $B\neq \emptyset$}
Let us first consider the easier case where there is a vertex $u\in B$.
Let $N(u)=\{v_1,\dots, v_t\}$ be the neighbors of $u$ in $G$.
We remove $u$ and insert the edge set $K(N(u))$, that is, the edge set of the complete graph on $N(u)$.
Let $H=(V-u, E\setminus \delta_G(u)\cup K(N(u)))$ be the resulting graph with the bipartition $\{B-u, R\}$ of $V-u$.

Let $\{E_1^*,\dots, E_k^*\}$ be the $P$-connected component decomposition of $E(H)$ in ${\cal PM}_f(H)$.
By Lemma~\ref{lemma:minimizer0}, $\{E_1^{*},\dots, E_k^{*}\}$ is a minimizer of the right hand side of (\ref{eq:main2}) for $E(H)$.
By induction, we have 
\begin{equation}
\label{eq:case1_2}
\rank(\overline{\{A_e(\br) : e\in E(H)\}})=\mbox{$\sum_{i=1}^k f(E_i^*)$}
\end{equation}
for almost all rod-configurations $\br:R\rightarrow Gr(d-1,W)$.
%Notice that, since $\{E_1,\dots, E_k\}$ satisfies (\ref{eq:easy_direction}) with the equality,
%we actually have
%\begin{equation}
%\label{eq:ase1_3}
%\rank(\overline{\{A_e(\br):e\in E_i^*\}})=f(E_i^*)
%\end{equation}
%for each $i=1,\dots, k$. 

If $N(u)=\{v\}$ for some $v\in V$, then $E=E(H)+uv$.
It is easy to see ${\cal A}(\br)=\overline{\{A_e(\br) : e\in E(H)\}}\oplus A_{uv}(\br)$,
and hence $\rank(\overline{{\cal A}(\br)})=\rank(\overline{\{A_e(\br) : e\in E(H)\}})+\rank(A_{uv}(\br))=\sum_{i=1}^kf(E_i^*)+f(\{uv\})$,
implying ``$\geq$'' direction of (\ref{eq:main2}) since $\{E_1^*,\dots, E_k^*,\{uv\}\}$ is a partition of $E$.

Thus, let us assume $|N(u)|\geq 2$.
Since $K(N(u))$ is a clique in $H$, it is straightforward to check that $K(N(u))$ is $P$-connected in ${\cal PM}_f(H)$, and hence
a $P$-connected component, say $E_k^*$, contains $K(N(u))$ as a subset. 
This implies 
\begin{equation}
\label{eq:case1_1}
f(E_k^*\setminus K(N(u))\cup \delta_G(u))=f(E_k^*)+D.
\end{equation}

Observe that, for any $vw\in K(N(u))$, we have 
\begin{equation}
\label{eq:contain}
A_{vw}(\br)\subseteq \overline{A_{vu}(\br) \cup A_{uw}(\br)}.
\end{equation}
Indeed, any element of $A_{vw}(\br)$ is written as 
\begin{equation*}
[\Bvector{}{0}, \Bvector{\cdots}{\cdots}, \Bvector{}{0},
\Bvector{v}{\bal}, \Bvector{}{0}, \Bvector{\cdots}{\cdots}, \Bvector{}{0},
\Bvector{w}{-\bal}, \Bvector{}{0},\Bvector{\cdots}{\cdots}, \Bvector{}{0}],
\end{equation*}
for some $\bal\in \mathbb{P}(\bigwedge^2 W)\cap H_{\br}(v)\cap H_{\br}(w)$.
This can be decomposed as
\begin{equation*}
[\Bvector{}{0}, \Bvector{\cdots}{\cdots}, \Bvector{}{0},
\Bvector{v}{\bal}, \Bvector{}{0}, \Bvector{\cdots}{\cdots}, \Bvector{}{0},
\Bvector{u}{-\bal}, \Bvector{}{0},\Bvector{\cdots}{\cdots}, \Bvector{}{0}]+
[\Bvector{}{0}, \Bvector{\cdots}{\cdots}, \Bvector{}{0},
\Bvector{u}{\bal}, \Bvector{}{0}, \Bvector{\cdots}{\cdots}, \Bvector{}{0},
\Bvector{w}{-\bal}, \Bvector{}{0},\Bvector{\cdots}{\cdots}, \Bvector{}{0}],
\end{equation*}
where these two terms are contained in $A_{vu}(\br)$ and $A_{uw}(\br)$, respectively, because $H_{\br}(u)=\mathbb{P}(\bigwedge^2 W)$ by $u\in B$.  

(\ref{eq:contain}) implies $\overline{\{A_e(\br) : e\in E(H)\}}\subseteq \overline{{\cal A}(\br)}$.
Moreover, we can always take independent $D$ points $p_1,\dots, p_D$ from $\{A_e(\br) : e\in \delta(u)\}$ since $u\in B$ and $|N(u)|\geq 2$.
Note that they always satisfy $\overline{\{p_1,\dots, p_D\}}\cap \overline{\{A_e(\br) : e\in E(H)\}}=\emptyset$ since $u\notin V(H)$.
We thus obtain
\begin{equation}
\label{eq:case1_3}
\rank(\overline{{\cal A}(\br)})\geq\rank(\overline{\{A_e(\br) : e\in E(H)\}})+D.
\end{equation}
Combining (\ref{eq:case1_1}), (\ref{eq:case1_2}), and (\ref{eq:case1_3}), we obtain 
$\rank(\overline{{\cal A}(\br)})\geq \sum_{i=1}^{k-1}f(E_i^*)+f(E_k^*\setminus K(N(u))\cup \delta(u))$,
implying ``$\geq$'' direction of (\ref{eq:main2}) since $\{E_1^*,\dots, E_{k-1}^*, E_k^*\setminus K(N(u))\cup \delta(u)\}$ is a partition of $E$.
This completes the proof for case  $B\neq \emptyset$.

\subsubsection{Case of $B=\emptyset$}

For any $u\in V$, let ${\cal P}_u=\{B+u,R-u\}$. 
Note that, by induction, the linear polymatroid ${\cal PM}({\cal A}(\br'))$ is equal to ${\cal PM}_f(G,{\cal P}_u)$ 
for almost all rod-configurations $\br'$ on $R-u$.
Our proof is based on this inductive relation.
Intuitively speaking, we will replace a body associated with $u$  by a rod $\br(u)$.
This operation corresponds with restricting $\mathbb{P}(V_u)=\mathbb{P}(\bigwedge^2 W)$ 
to a hyperplane $H_{\br}(u)$ of $\mathbb{P}(\bigwedge^2 W)$, which is the dual of the point $\br(u)$. 
This operation is equivalent to the restriction of ${\cal A}(\br')$ to a special hyperplane $H$ in $\mathbb{P}(V_V)$ 
such that $H\cap \mathbb{P}(V_u)=H_{\br}(u)$ and $\mathbb{P}(V_v)\subset H$ for all $v\in V-u$.
This hyperplane $H$ is not generic within $\mathbb{P}(V_V)$ (and hence this operation is not Dilworth truncation), 
but we may take $H$ so that  $H\cap \mathbb{P}(V_u)$ is generic within $\mathbb{P}(V_u)$.
We will show that the naturally extended rank formula of Dilworth truncation holds for this operation for some $u\in V$.

The proof consists of sequence of lemmas.
We first define a generic hyperplane within $\mathbb{P}(V_u)$ for a vertex $u\in V$
and show the existence of generic hyperplanes in Lemma~\ref{lemma:existence}.
We then discuss about an extension of a rod-configuration $\br':R-u\rightarrow Gr(d-1,W)$ to $\br:R\rightarrow Gr(d-1,W)$,
where $\br$ is said to be an {\em extension} of $\br'$ if $\br(v)=\br'(v)$ for all $v\in V-u$.
We shall define a generic extension of a rod-configuration based on a generic hyperplane in  $\mathbb{P}(V_u)$.  
Then in Lemma~\ref{lemma:combinatorial} we shall show an existence of a vertex $u\in V$ having special properties 
and finally perform a variant of Dilworth truncation at $u$ in Lemma~\ref{lemma:last_truncation}.

For a flat $A$ of $\mathbb{P}(V_V)$ and a vertex $u\in V$, ${\rm proj}_u(A)$ denotes 
the orthogonal projection\footnote{More precisely, let $W'$ be the linear subspace of $V_V$ 
satisfying $A=\mathbb{P}(W')$, 
and let ${\rm proj}_u(W')$ be the orthogonal projection of $W'$ onto $V_u$. 
We define ${\rm proj}_u(A)$ by $\mathbb{P}({\rm proj}_u(W'))$.} of $A$ onto $\mathbb{P}(V_u)$.
A hyperplane $H_u$ of $\mathbb{P}(V_u)$ 
is called {\em generic relative to} a finite set ${\cal A}$ of flats in $\mathbb{P}(V_V)$  if it satisfies
the following property;
for every ${\cal A}_1, {\cal A}_2\subseteq {\cal A}$ 
with ${\rm proj}_u(\overline{{\cal A}_1}\cap \overline{{\cal A}_2})\neq \emptyset$ (where we allow ${\cal A}_1={\cal A}_2$),
\begin{equation}
\label{eq:v_generic0}
\rank ({\rm proj}_u(\overline{{\cal A}_1}\cap \overline{{\cal A}_2})\cap H_u) 
= \rank ({\rm proj}_u(\overline{{\cal A}_1}\cap \overline{{\cal A}_2}))-1.
\end{equation}
The next lemma shows the existence of generic hyperplanes.
\begin{lemma}
\label{lemma:existence}
Let $u\in V$ and ${\cal A}$ be a finite set of flats in $\mathbb{P}(V_V)$.
Suppose $Gr(d-1,W)\subseteq \mathbb{P}(V_u)$ (by identifying $V_u$ with $\bigwedge^{d-1} W$).
Then, for almost all points $[\br_u]\in Gr(d-1,W)$,
the hyperplane $H_u$ of  $\mathbb{P}(V_u)$ dual to $[\br_u]$ is generic relative to 
${\cal A}$.
\end{lemma}
\begin{proof}
Take any ${\cal A}_1, {\cal A}_2\subseteq {\cal A}$ with ${\rm proj}_u(\overline{{\cal A}_1}\cap \overline{{\cal A}_2})\neq \emptyset$,
and let us denote $A=\overline{{\cal A}_1}\cap \overline{{\cal A}_2}$ for simplicity.
It is clear that $\rank ({\rm proj}_u(A)\cap H_u)\geq\rank({\rm proj}_u(A)) -1$ for any $H_u$. 
Let us consider the ``$\leq$'' direction.
If ${\rm proj}_u(A)=\mathbb{P}(V_u)$, this relation clearly holds.
Otherwise ${\rm proj}_u(A)$ is a linear subspace of $\mathbb{P}(V_u)$, and hence
$\rank({\rm proj}_u(A)\cap H_u)\leq \rank({\rm proj}_u(A))-1$ holds 
if we take $[\br_u]\in Gr(d-1,W)$ so that $[\br_u]$ is not contained 
in the dual of ${\rm proj}_u(A)$ in $\mathbb{P}(V_u)$.
Since the intersection of the dual of  ${\rm proj}_u(A)$ with $Gr(d-1,W)$ 
is a lower dimensional subvariety of $Gr(d-1,W)$, almost all $[\br_u]$ satisfy this property. 

Since there are a finite number of possible $A=\overline{{\cal A}_1}\cap \overline{{\cal A}_2}$,
almost all hyperplanes $H_u$ of $\mathbb{P}(V_u)$ are indeed generic.
\end{proof}

We now define a {\em generic extension} of a rod-configuration $\br':R-u\rightarrow Gr(d-1,W)$ as follows:
a rod-configuration $\br:R\rightarrow Gr(d-1,W)$ is a generic extension of $\br'$ if
\begin{description}
\item[(Condition for extension):] $\br(v)=\br'(v)$ for $v\in V-u$; 
\item[(Condition for genericity):] $\br(u)$ satisfies the property that 
the dual hyperplane of $\br(u)$ in $\mathbb{P}(V_u)$ is generic relative to ${\cal A}(\br')$.
\end{description}
By Lemma~\ref{lemma:existence}, almost all extensions are generic.

Once we pick out a generic extension $\br$ of $\br'$, 
the unique hyperplane $H$ of $\mathbb{P}(V_V)$ is determined in such a way that
$H\cap \mathbb{P}(V_u)$ is the dual hyperplane of $\br(u)$ in $\mathbb{P}(V_u)$ 
and $\mathbb{P}(V_v)\subset H$ for all $v\in V-u$.
Such a unique hyperplane $H$ is called the {\em hyperplane associated with the generic extension}.

It is important to observe 
\begin{equation}
\label{eq:flat_restriction}
A_e(\br)=A_e(\br')\cap H \qquad \text{for every } e\in E.
\end{equation}
Also, if we define $\chi_u$ by 
\[
\chi_u(A)=\begin{cases} 1 & \text{if ${\rm proj}_u(A) \neq \emptyset$} \\
0 & \text{otherwise}
\end{cases} 
\]
for a flat $A\subset \mathbb{P}(V_V)$,
then we have the following from the genericity (\ref{eq:v_generic0}):
for every ${\cal A}_1, {\cal A}_2\subseteq {\cal A}(\br')$, 
\begin{equation}
\label{eq:v_generic1}
\rank ((\overline{{\cal A}_1}\cap \overline{{\cal A}_2})\cap H) = \rank (\overline{{\cal A}_1}\cap \overline{{\cal A}_2})
-\chi_u(\overline{{\cal A}_1}\cap \overline{{\cal A}_2}).
\end{equation}
Note that, setting ${\cal A}_1={\cal A}_2$, (\ref{eq:v_generic1}) implies, for every ${\cal A}_1\subseteq {\cal A}(\br')$
\begin{equation}
\label{eq:v_generic2}
\rank (\overline{{\cal A}_1}\cap H)=\rank(\overline{{\cal A}_1})-\chi_u(\overline{{\cal A}_1}).
\end{equation}
In particular, for any $A_e(\br')\in {\cal A}(\br')$, 
\begin{equation}
\label{eq:v_generic3}
\rank (A_e(\br')\cap H)=\rank(A_e(\br'))-\chi_u(A_e(\br')).
\end{equation}

By (\ref{eq:flat_restriction}), our goal is now to extend Theorem~\ref{theorem:truncation} 
to the case of our special hyperplane $H$.
Such an extension will be given in Lemma~\ref{lemma:last_truncation} 
by performing a truncation at a vertex $u$ shown in the following lemma (Lemma~\ref{lemma:combinatorial}).
Before that, we need an easy observation.  

\begin{lemma}
\label{lemma:flat_connected}
Let  $C$ be a $P$-connected set in ${\cal PM}(G,{\cal P})$ with $C\neq E$.
Then, for almost all rod-configurations $\br$ on $R$,
${\cal A}_C(\br)$ is connected.
\end{lemma}
\begin{proof}
Let us consider the restriction to $C$, i.e., consider $G'=(V(C),C)$, ${\cal P}'=\{B\cap V(C), R\cap V(C)\}$.
Note that $|V(C)|\leq |V|, |R\cap V(C)|\leq |R|$ and $|C|<|E|$ by $C\subseteq E-e$.
Hence, by induction, the linear polymatroid ${\cal PM}({\cal A}_C(\br))$ is equal to ${\cal PM}_f(G',{\cal P}')$ for almost all rod-configurations $\br$ on $R$.
Since $C$ is $P$-connected in ${\cal PM}_f(G',{\cal P}')$, ${\cal A}_C(\br)$ is connected. 
\end{proof}

\begin{lemma}
\label{lemma:combinatorial}
There exists a vertex $u$ satisfying one of the following two properties:
For almost all rod-configurations $\br'$ on $R-u$ and 
almost all extension $\br$,
\begin{description}
\item[(A)] $G$ has an edge subset $C$ 
with $\delta_G(u)\subset C\subsetneq E$ such that ${\cal A}_C(\br)$ is connected; or 
\item[(B)] $G$ has disjoint edge subsets $C$ and $C'$ with $\delta_G(u)\subset C\cup C'$ such that 
both ${\cal A}_C(\br)$ and ${\cal A}_{C'}(\br)$ are connected. 
Furthermore, if ${\cal A}(\br')$ is connected, then  ${\rm proj}_u(\overline{{\cal A}_C(\br')}\cap \overline{{\cal A}_{E\setminus C}(\br')})\neq \emptyset$.
\end{description}
\end{lemma}
\begin{proof}
Take any edge $e\in E$, and consider $G-e$.
By Lemma~\ref{lemma:degree},  $G-e$ has 
(i) three vertices each of which is spanned by two $P$-connected components of ${\cal PM}(G-e,{\cal P})$, or
(ii) a vertex spanned by exactly one $P$-connected component of ${\cal PM}(G-e,{\cal P})$.
Since any $P$-connected set of ${\cal PM}(G-e,{\cal P})$ is also $P$-connected in ${\cal PM}(G,{\cal P})$, these $P$-connected components are $P$-connected in ${\cal PM}(G,{\cal P})$.

We define $C,C'\subseteq E-e$ as follows:
If (i) occurs, then take a vertex $u$ 
that is not an endpoint of $e$ and is spanned by two $P$-connected components in ${\cal PM}(G-e,{\cal P})$. 
Let $C$ and $C'$ be such components.
If (ii) occurs, then we have a vertex $u$ spanned by exactly one $P$-connected component in ${\cal PM}(G-e,{\cal P})$.
Let $C$ be that component. 
Furthermore, if $u$ is an endpoint of $e$, let $C'=\{e\}$. 

Consequently, one of the followings holds: (i') $C$ is $P$-connected set with $\delta_G(u)\subset C\subsetneq E$ or 
(ii') $C$ and $C'$ are disjoint $P$-connected sets (that may be trivial) with $\delta_G(u)\subset C\cup C'$.
Note that, both $C$ and $C'$ are proper subsets of $E$, and thus Lemma~\ref{lemma:flat_connected} implies that  ${\cal A}_C(\br)$ and ${\cal A}_{C'}(\br)$ are connected 
for almost all rod-configurations $\br$.

The remaining thing is to prove the last property of (B) when (ii') occurs.
Recall ${\cal P}_u=\{B+u,R-u\}$, 
and the linear polymatroid ${\cal PM}(\{A_e(r'):e\in K(V)\})$ 
is equal to ${\cal PM}_f(K(V),{\cal P}_u)$ by induction on the lexicographical order of $(|V|,|R|,|E|)$.
Since ${\cal A}(\br')$ is connected, $E$ is $P$-connected in ${\cal PM}_f(G,{\cal P}_u)$.
Thus, applying Lemma~\ref{lemma:stronger}, we may assume that 
there is a vertex $v\in V-u$ with $uv\in {\rm cl}(C)\cap {\rm cl}(E\setminus C)$ 
for the closure operator of ${\cal PM}_f(G,{\cal P}_u)$.
This implies $A_{uv}(\br')\subset \overline{{\cal A}_C(\br')}\cap \overline{{\cal A}_{E\setminus C}(\br')})$,
and thus ${\rm proj}_u(\overline{{\cal A}_C(\br')}\cap \overline{{\cal A}_{E\setminus C}(\br')})\neq \emptyset$.
\end{proof}

We are now ready to extend Theorem~\ref{theorem:truncation} to our nongeneric hyperplane.
Recall that, for a family ${\cal A}$ of flats and a hyperplane $H$, we abbreviate
$\{A\cap H :  A\in {\cal A}\}$ as ${\cal A}\cap H$.
Note that $(\overline{{\cal A}})\cap H$ implies $\overline{\{A :  A\in {\cal A}\}}\cap H$, 
which may not be equal to $\overline{{\cal A}\cap H}=\overline{\{A\cap H :  A\in {\cal A}\}}$.

\begin{lemma}
\label{lemma:last_truncation}
Let $u$ be a vertex shown in Lemma~\ref{lemma:combinatorial} and $\br'$
be a generic  rod-configuration on $R-u$.
Then, for the hyperplane $H$ of $\mathbb{P}(V_V)$ associated with a generic extension of $\br'$, 
\begin{equation}
\label{eq:v_truncation}
\rank(\overline{{\cal A}(\br')\cap H}) = 
\min\{\mbox{$\sum_{i=1}^k$} (\rank(\overline{{\cal A}_{E_i}(\br')})-\chi_u(\overline{{\cal A}_{E_i}(\br')}))\},
\end{equation}
where the minimum is taken over all partitions 
$\{E_1,\dots, E_k\}$ of $E$ into nonempty subsets. 
\end{lemma}
\begin{proof}
For simplicity, we abbreviate  $A_e(\br')$ as $A_e'$ and  ${\cal A}(\br')$ as  ${\cal A}'$, respectively.
Consider the connected component decomposition of ${\cal A}'\cap H$ 
(that is, the $P$-connected component decomposition of the linear polymatroid ${\cal PM}({\cal A}'\cap H)$). 
To see the equality of (\ref{eq:v_truncation}), 
we show (\ref{eq:v_truncation}) for each connected component of ${\cal A}'\cap H$
Thus, by induction, we may assume ${\cal A}'\cap H$ is connected and it is sufficient to show
\begin{equation}
\label{eq:main_claim}
\rank(\overline{{\cal A}'\cap H})=\rank(\overline{{\cal A}'})-1.
\end{equation}

From the choice of $u$, (A) or (B) of Lemma~\ref{lemma:combinatorial} holds.
Let $C$ and $C'$ be subsets of $E$ satisfying properties  of Lemma~\ref{lemma:combinatorial},
where $C'=\emptyset$ if (A) holds (otherwise we may assume $C'\neq \emptyset$).
Namely, if $C'=\emptyset$, ${\cal A}_{C}'\cap H$  is connected with $\delta_G(u)\subseteq C$.
(Note that, in the current situation,  
${\cal A}_{C}'\cap H$ corresponds to ${\cal A}_C(\br)$ of the statement of Lemma~\ref{lemma:combinatorial}).)
If $C'\neq \emptyset$, 
${\cal A}_{C}'\cap H$ and ${\cal A}_{C'}'\cap H$ are connected with $\delta_G(u)\subseteq C\cup C'$. 
We may further assume $\delta_G(u)\cap C\neq \emptyset$ and 
 $\delta_G(u)\cap C'\neq \emptyset$, since otherwise we have the former case. 

We now calculate the rank of 
$\overline{{\cal A}_{C}'\cap H}$, $\overline{{\cal A}_{C'}'\cap H}$, and $\overline{{\cal A}_{E\setminus C}'\cap H}$.
The connectivity of ${\cal A}_{C}'\cap H$ and $\delta_G(u)\cap C$ imply 
\begin{equation}
\label{eq:last1}
\rank(\overline{{\cal A}_{C}'\cap H})=\rank(\overline{{\cal A}_C'})-1
\end{equation}
by induction. Similarly, if $C'\neq \emptyset$, 
the connectivity of ${\cal A}_{C'}'\cap H$ and $\delta_G(u)\cap C'$ imply 
\begin{equation}
\label{eq:last2}
\rank(\overline{{\cal A}_{C'}'\cap H})=\rank(\overline{{\cal A}_{C'}'})-1.
\end{equation}
Also, since all flats of ${\cal A}_{E\setminus (C\cup C')}'$ are contained in $H$
by $\delta_G(u)\subset C\cup C'$, we have
\begin{equation}
\label{eq:last3}
{\cal A}_{E\setminus (C\cup C')}'\cap H={\cal A}_{E\setminus (C\cup C')}'.
\end{equation}

Suppose $C'\neq \emptyset$,
and let us take an edge $e\in \delta_G(u)\cap C'$ and a point 
$x\in A_e'\setminus H$. (Note that, by (\ref{eq:v_generic3}), $A_e'\setminus H\neq \emptyset$.)
Then, clearly 
$\rank(\overline{\overline{({\cal A}_{C'}'\cap H)}\cup \{x\}})=\rank(\overline{{\cal A}_{C'}'\cap H})+1$.
Combined with  (\ref{eq:last2}), we have
\begin{equation}
\label{eq:last4}
\overline{{\cal A}_{C'}'}=\overline{\overline{({\cal A}_{C'}'\cap H)}\cup \{x\}}.
\end{equation}
By (\ref{eq:last3}) and (\ref{eq:last4}),
\begin{align}
\overline{{\cal A}_{E\setminus C}'}\cap H
&=\overline{{\cal A}_{E\setminus (C\cup C')}'\cup {\cal A}_{C'}'} \cap H  \nonumber \\ 
&=\overline{(\overline {{\cal A}_{E\setminus (C\cup C')}'\cap H}) 
\cup (\overline{{\cal A}_{C'}'\cap H})\cup \{x\}}\cap H \nonumber \\
&=\overline{(\overline{{\cal A}_{E\setminus (C\cup C')}'\cap H})\cup (\overline{{\cal A}_{C'}'\cap H})} 
=\overline{{\cal A}_{E\setminus C}'\cap H}. \label{eq:last43}
\end{align}
Thus, applying (\ref{eq:v_generic2}), we obtain
\begin{equation}
\label{eq:last42}
\rank(\overline{{\cal A}_{E\setminus C}'\cap H})
=\rank(\overline{{\cal A}_{E\setminus C}'}\cap H)=\rank(\overline{{\cal A}_{E\setminus C}'})-1
\end{equation}
if $C'\neq \emptyset$. In total, combining (\ref{eq:last3}) and (\ref{eq:last42}), 
\begin{equation}
\label{eq:last23}
\rank(\overline{{\cal A}_{E\setminus C}'\cap H})=
\begin{cases}
\rank(\overline{{\cal A}_{E\setminus C}'}) -1 & \text{ (if  $C'\neq \emptyset$)} \\
\rank(\overline{{\cal A}_{E\setminus C}'})  & \text{ (if  $C'=\emptyset$)}. 
\end{cases}
\end{equation}

We then compute the rank of $(\overline{{\cal A}_C'\cap H})\cap (\overline{{\cal A}_{E\setminus C}'\cap H})$.
Since $\rank((\overline{{\cal A}_{C}'})\cap H)=\rank(\overline{{\cal A}_{C}'})-1$ by (\ref{eq:v_generic2}),
comparing this relation with (\ref{eq:last1}), we have 
\begin{equation}
\label{eq:last5}
\overline{{\cal A}_C'}\cap H=\overline{{\cal A}_C'\cap H}.
\end{equation}
%Symmetrically, if $C'\neq \emptyset$, we have 
%\begin{equation}
%\label{eq:last6}
%\overline{{\cal A}_{C'}'}\cap H=\overline{{\cal A}_{C'}'\cap H}.
%\end{equation}
By (\ref{eq:last3}), (\ref{eq:last43}) and (\ref{eq:last5}), 
we obtain
\begin{equation}
\label{eq:last7}
(\overline{{\cal A}_C'\cap H})\cap (\overline{{\cal A}_{E\setminus C}'\cap H})
=(\overline{{\cal A}_C'}) \cap (\overline{{\cal A}_{E\setminus C}'}) \cap H. 
\end{equation}
Therefore, applying (\ref{eq:last7}) and then (\ref{eq:v_generic1}), we obtain
\begin{align}
\nonumber
\rank((\overline{{\cal A}_C'\cap H})\cap (\overline{{\cal A}_{E\setminus C}'\cap H}))
&=\rank((\overline{{\cal A}_C'}) \cap (\overline{{\cal A}_{E\setminus C}'}) \cap H)  \\
&=\rank(\overline{{\cal A}_C'} \cap \overline{{\cal A}_{E\setminus C}'} )
-\chi_u(\overline{{\cal A}_C'} \cap \overline{{\cal A}_{E\setminus C}'}).  \label{eq:last8}
\end{align}
We show that 
$\chi_u(\overline{{\cal A}_C'} \cap \overline{{\cal A}_{E\setminus C}'})$ takes distinct values 
depending on whether $C'=\emptyset$ or not.
If $C'=\emptyset$, then no edge in $E\setminus C$ is incident to $u$ by $\delta_G(u)\subseteq C$,
and ${\rm proj}_u (\overline{{\cal A}_{E\setminus C}'})=\emptyset$. 
Thus, $\chi_u(\overline{{\cal A}_C'} \cap \overline{{\cal A}_{E\setminus C}'})=0$.
On the other hand, if $C'\neq \emptyset$, then 
property (B) of Lemma~\ref{lemma:combinatorial} 
implies ${\rm proj}_u(\overline{{\cal A}_C'}\cap \overline{{\cal A}_{E\setminus C}})\neq \emptyset$
since ${\cal A}'$ is connected from the connectivity of ${\cal A}'\cap H$.
Therefore, $\chi_u(\overline{{\cal A}_C'} \cap \overline{{\cal A}_{E\setminus C}'})=1$ if $C'\neq \emptyset$.
In total, (\ref{eq:last8}) can be rewritten by
\begin{equation}
\label{eq:last9}
\rank((\overline{{\cal A}_C'\cap H})\cap (\overline{{\cal A}_{E\setminus C}'\cap H}))
=
\begin{cases}
\rank(\overline{{\cal A}_C'} \cap \overline{{\cal A}_{E\setminus C}'} )-1 & \text{ (if $C'\neq \emptyset$)} \\
\rank(\overline{{\cal A}_C'} \cap \overline{{\cal A}_{E\setminus C}'} ) & \text{ (if $C'=\emptyset$)}.
\end{cases}
\end{equation}

By (\ref{eq:last1}), (\ref{eq:last23}), (\ref{eq:last9}), and the modularity of $\rank(\cdot)$, 
\begin{align*}
\rank(\overline{{\cal A}'\cap H})&=\rank(\overline{(\overline{{\cal A}_{C}'\cap H})\cup (\overline{{\cal A}_{E\setminus C}'\cap H})})\\
&=\rank(\overline{{\cal A}_{C}'\cap H})+\rank(\overline{{\cal A}_{E\setminus C}'\cap H})-\rank((\overline{{\cal A}_C'\cap H})\cap (\overline{{\cal A}_{E\setminus C}'\cap H})) \\
&=\rank(\overline{{\cal A}_{C}'})+\rank(\overline{{\cal A}_{E\setminus C}'})-\rank(\overline{{\cal A}_C'}\cap \overline{{\cal A}_{E\setminus C}'})-1 \\
&=\rank(\overline{{\cal A}'})-1,
\end{align*}
implying (\ref{eq:main_claim}).
This completes the proof of the lemma.
\end{proof}

For $F\subseteq E$, let
\[
\chi_u(F)=\begin{cases} 1 & \text{if $u\in V(F)$} \\
0 & \text{otherwise},
\end{cases} 
\]
where $u$ is a vertex shown in Lemma~\ref{lemma:combinatorial}.
Then, Lemma~\ref{lemma:last_truncation} implies,
for almost all bar-configurations $\br'$ on $R-u$ and its generic extension $\br$,
\begin{equation}
\label{eq:final1}
\rank(\overline{{\cal A}(\br)})=\mbox{$\min\{\sum_i(\rank(\overline{{\cal A}_{E_i}(\br')})-\chi_u(E_i)) :$  a partition $\{E_1,\dots, E_k\}$ of $E\}$}. 
\end{equation}
Let $R'=R-u$.
The induction hypothesis on $|R|$ implies 
\begin{equation}
\label{eq:final2}
\rank(\overline{{\cal A}_{E_i}(\br')})=\mbox{$\min\{\sum_j(D|V(E_{i,j})|-D-|R'(E_{i,j})|) :$  a partition $\{E_{i,1},\dots, E_{i,k'}\}$ of $E_i\}$}
\end{equation}
for each $E_i\subseteq E$.
Since $\chi_u(E_i)\leq \sum_j\chi_u(E_{i,j})$ for any $E_i\subseteq E$ and any partition $\{E_{i,1},\dots, E_{i,k'}\}$ of $E_i$,
(\ref{eq:final1}) and (\ref{eq:final2}) imply
\begin{equation}
\label{eq:final3}
\rank(\overline{{\cal A}(\br)})\geq \mbox{$\min\{\sum_{i}(D|V(E_{i})|-D-|R'(E_{i})|-\chi_u(E_{i}))  :$  a partition $\{E_1,\dots, E_k\}$ of $E\}$}. 
\end{equation}
%&\geq \mbox{$\min\{\sum_i(D|V(E_{i})|-D-|R'(E_{i})|-\delta_u(E_i)):$ a partition $\{E_1,\dots, E_k\}$ of $E\}$} \\
%&=\mbox{$\min\{\sum_i (D|V(E_{i})|-D-|R(E_{i})|)):$ a partition $\{E_1,\dots, E_k\}$ of $E\}$},
%\end{align*}
Note that, for any $F\subseteq E$, we have $|R(F)|=|R'(F)|+\chi_u(F)$.
Thus, (\ref{eq:final3}) implies ``$\geq$'' direction of (\ref{eq:main2}) for case $B=\emptyset$.
This completes the proof of Theorem~\ref{theorem:main2}.
\end{proof}

\section{Identified Body-hinge Frameworks}
\label{sec:body_hinge}
An {\em identified body-hinge framework} (simply called a body-hinge framework) is a
structure consisting of rigid bodies connected by hinges (that is, $(d-2)$-dimensional flats).
A hinge allows to connect any number of bodies.
A body-hinge framework is formally defined as a pair $(G,\bh)$, where
\begin{itemize}
\item $G=(B,H;E)$ is a bipartite graph with vertex classes $B$ and $H$, representing bodies and hinges, respectively;
\item $\bh:H \rightarrow Gr(d-1,W)$ is a hinge-configuration.
\end{itemize}
Note that each $v_1\in B$ and $v_2\in H$ correspond to a body and a hinge, respectively, and $e\in E$ indicates their incidence.
%Although we can introduce external constraints as above, we would like to concentrate on those without external constraints for simplicity;
%the goal of this section is to present a simpler proof of Tay's theorem.
%(But we will make use of external constraints in the proof.)

A motion of $(G,\bh)$ is defined as a mapping $\bmm:B\rightarrow \bigwedge^{d-1} W$ 
such that $\bmm(u)-\bmm(v)$ is contained in 
$\bh(w)$ for any neighbors $u,v\in B$ of $w\in H$.
A motion $\bmm$ is called {\em trivial} if $\bmm(v)$'s are equal for all $v\in B$.
$(G,\bh)$ is said to be {\em infinitesimally rigid} if every motion is trivial.

For a bipartite graph $G=(B,H;E)$, 
the graph obtained from $G$ by duplicating each edge by $(D-1)$ parallel copies is denoted by $(D-1)\circ G$, 
and $(D-1)\circ E$ denotes the edge set of $(D-1)\circ G$.
Tay showed a combinatorial characterization of identified body-hinge frameworks
by converting to rod-bar frameworks.
Below, we give a more natural proof.
\begin{corollary}[Tay~\cite{tay:89}]
\label{coro:body_hinge}
Let $G=(B,H;E)$ be a bipartite graph. Then, there exists a hinge-configuration $\bh$
such that $(G,\bh)$ is infinitesimally rigid if and only if 
$(D-1)\circ G$ contains an edge subset $I\subseteq (D-1)\circ E$ satisfying the following counting conditions:
\begin{itemize}
\item $|I|=D|B|+(D-1)|H|-D$;
\item $|F|\leq D|B(F)|+(D-1)|H(F)|-D$ for each nonempty $F\subseteq I$.
\end{itemize}
\end{corollary}
\begin{proof}
Let $(G,\bh)$ be an identified body-hinge framework.
For an edge $e=uv\in E$ with $u\in H$ and $v\in B$,
we can regard $\bh(u)$ as a rod (generically) linked by $(D-1)$ bars with the body associated with $v$ (see Figure~\ref{fig:hinge_to_rod}).
Hence, the identified body-hinge framework $(G,\bh)$ is equal to 
the body-rod-bar framework $(G',\bq,\br)$,
where $G'$ is the graph with $V(G')=B\cup H$ and $E(G')=(D-1)\circ E$, $\br=\bh$, and  $\bq$ is a generic bar-configuration.
Since $D(|B(F)\cup H(F)|-1)-|H(F)|=D|B(F)|+(D-1)|H(F)|-D$ for any $F\subseteq (D-1)\circ E$,
the statement follows from Theorem~\ref{theorem:main}.
\end{proof}
The proof can be extended to frameworks consisting of bodies, rods, bars, and hinges without difficulty.

\begin{figure}[t]
\centering
\includegraphics[scale=0.8]{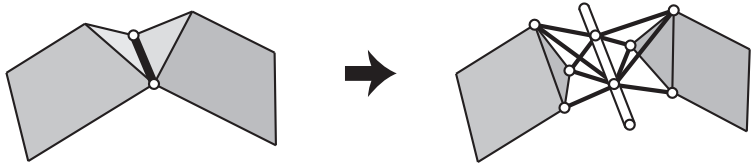}
\caption{Conversion of the body-hinge model to the body-rod-bar model.}
\label{fig:hinge_to_rod}
\end{figure}

Katoh and Tanigawa~\cite{molecular} showed that,
if each hinge is allowed to connect only two bodies, then each body can be realized as a rigid panel (i.e., a hyperplane).
Namely, a panel-hinge framework, which consists of rigid panels connected by hinges, 
is generically characterized by the counting condition of Corollary~\ref{coro:body_hinge}.
A natural question is whether we can drop the restriction or not.
\begin{problem}
Let $G=(B,H;E)$ be a bipartite graph  satisfying the counting condition of Corollary~\ref{coro:body_hinge}.
Is there a hinge-configuration $\bh$ such that $(G,\bh)$ is an infinitesimally rigid panel-hinge framework?
\end{problem}
Indeed, this problem was already discussed in, e.g.,~\cite{tay:whiteley:84,tay:89,Whiteley:89}
and is unsolved even for $2$-dimensional case.
In~\cite{Whiteley:89}, Whiteley presented a partial solution  for $2$-dimensional case. 

In the context of combinatorial rigidity, three types of characterizations are typically considered;
Maxwell/Laman-type counting conditions, Henneberg-type graph constructions, and tree-decompositions.
In particular, tree-decompositions often  provide very short proofs for combinatorial characterizations. See, e.g.,~\cite{whiteley:88,tay:93,Jackson:07}.
It is hence natural to ask a tree-decomposition for identified body-hinge frameworks, which leads to Corollary~\ref{coro:body_hinge}.
\begin{problem}
Let $G=(B,H;E)$ be a bipartite graph.
Suppose there is an edge set $I\subseteq (D-1)\circ E$ such that
$|I|=D|B|+(D-1)|H|-D$ and $|F|\leq D|B(F)|+(D-1)|H(F)|-D$ for each nonempty $F\subseteq I$.
Then, does $(D-1)\circ G$ contain $D$ edge-disjoint trees such that
each vertex of $B$ is spanned by all of them and each vertex of $H$ is spanned by exactly $D-1$ trees among them.
\end{problem}
The problem may be false since the problem of deciding whether a hypergraph contains $k$ edge-disjoint spanning connected subgraphs
is NP-hard even for $k=2$~\cite{frank:2003}.

As for computational issue, $O(|V|^2)$ time algorithms are known for computing the rank of the counting (poly)matroids appeared in this paper 
(see, e.g.,~\cite{Imai:1983,gabow:1992,lee:streinu:2005,berg:jordan:2003b} for more detail).
Developing a sub-quadratic algorithm is indeed a challenging problem.

\section{Direction-rigidity}
\label{sec:direction}
As a direct application of Dilworth truncation, 
we shall briefly discuss direction-rigidity of bar-joint frameworks.

Recall that a {\em $d$-dimensional bar-joint framework} is a pair $(G,\bp)$, where $G=(V,E)$ is a graph and 
$\bp:V\rightarrow\mathbb{R}^d$. 
Each vertex represents a joint and each edge represents a bar
which usually constraints the distance between two endpoints.
As a variant of length-constraint, direction-constraint (and the mixture of length and direction constraints) 
has been considered in the literature (see, e.g.,~\cite{Whitley:1997,jackson2010globally,Servatius:1999}).
In~\cite{Whitley:1997}, Whiteley showed a combinatorial characterization of direction-rigidity 
as a corollary of a combinatorial characterization of reconstructivity of pictures appeared in scene analysis (see, e.g.,\cite{Whiteley:89,Whitley:1997,whiteley:hand}).
In this section we provide a direct proof of this characterization.

For a $d$-dimensional bar-joint framework $(G,\bp)$, 
an infinitesimal motion $\bmm:V\rightarrow \mathbb{R}^d$ of $(G,\bp)$ under direction-constraint
is an assignment of $\bmm(v)\in \mathbb{R}^d$ to each $v\in V$ such that
$\bmm(u)-\bmm(v)$ is parallel to  $\bp(u)-\bp(v)$ for any $uv\in E$,
i.e., $\bmm(u)-\bmm(v)=t(\bp(u)-\bp(v))$ for some $t\in \mathbb{R}$.
Of course, the direction-constraint for each $uv\in E$ can be written as
\begin{equation}
\label{eq:direction}
(\bmm(u)-\bmm(v))\cdot {\bm \alpha}=0 \qquad  \text{for any ${\bm \alpha}\in \mathbb{R}^d$ with $(\bp(u)-\bp(v))\cdot {\bm \alpha}=0$}. 
\end{equation}

It is easy to observe that the space of infinitesimal motions of $(G,\bp)$ has dimension at least $d+1$;
a linear combination of parallel transformations to $d$ directions and the dilation  centered at the origin 
(see, e.g.,~\cite[Section~8]{Whitley:1997} for more detail).
We say that $(G,\bp)$ is {\em direction-rigid} if the dimension of the motion space is exactly $d+1$.

In this section, we shall use $V_u$ to denote a $d$-dimensional vector space associated with $u$
(which was $D$-dimensional in the preceding sections),
and let $V_V$ denote the direct product of $V_u$ for all $u\in V$.
Hence, $V_V$ is $d|V|$-dimensional in this case.
For each $uv\in E$, let us define a $(d-2)$-dimensional flat of $\mathbb{P}(V_V)$ by 
\begin{equation}
\label{eq:direction_flat}
A_{uv}(\bp)=\{ 
[\Bvector{}{0}, \Bvector{}{\cdots}, \Bvector{}{0}, 
\Bvector{u}{\bm \alpha}, \Bvector{}{0}, \Bvector{}{\cdots}, \Bvector{}{0},
\Bvector{v}{-{\bm \alpha}}, \Bvector{}{0}, \Bvector{}{\cdots}, \Bvector{}{0}]
: {\bm \alpha}\in \mathbb{R}^d, \ (\bp(u)-\bp(v))\cdot {\bm \alpha}=0\},
\end{equation}
and let ${\cal A}(\bp)=\{A_e(\bp) : e\in E\}$.
Then, it is easy to see that direction-rigidity is characterized by the polymatroid ${\cal PM}({\cal A}(\bp))$
in the sense that $(G,\bp)$ is direction-rigid if and only if the rank of ${\cal PM}({\cal A}(\bp))$ is equal to $d|V|-(d+1)$.
The following theorem provides a combinatorial characterization of this polymatroid.

\begin{theorem}
\label{theorem:direction}
Let $f':2^E\rightarrow \mathbb{Z}$ be an integer-valued monotone submodular function defined by
\begin{equation}
f'(F)=d|V(F)|-(d+1) \qquad (F\subseteq E).
\end{equation}
Then, for almost all joint-configurations $\bp:V\rightarrow \mathbb{R}^d$, 
${\cal PM}({\cal A}(\bp))$ is equal to the polymatroid ${\cal PM}_{f'}(G)=(E,\hat{f'})$ induced by $f'$. 
\end{theorem}
\begin{proof}
We prove $\rank(\overline{{\cal A}_F(\bp)})=\hat{f'}(F)$ for any nonempty $F\subseteq E$ 
(see (\ref{eq:g_hat}) for the definition of $\hat{f}$).
%As usual, one direction is easy; for any $F\subseteq E$, $\rank(\overline{{\cal A}_F(p)})\leq \hat{f'}(F)$.
%We only prove the nontrivial direction ``$\geq$''.
The idea is exactly the same as the alternative proof of Laman's theorem by  Lov{\'a}sz and Yemini\cite{lovasz:1982}.

Recall that $g=|V(\cdot)|-1$ is the monotone submodular function inducing graphic matroid.
As mentioned in Section~\ref{subsec:graphic}, the union of $d$ copies of the graphic matroid is the matroid induced by $dg$ as well as 
the generic matroid associated with the family ${\cal A}=\{A_e: e\in E\}$ of flats
\begin{equation*}
A_{uv}=\{ 
[\Bvector{}{0}, \Bvector{}{\cdots}, \Bvector{}{0}, 
\Bvector{u}{\bm \alpha}, \Bvector{}{0}, \Bvector{}{\cdots}, \Bvector{}{0},
\Bvector{v}{-{\bm \alpha}}, \Bvector{}{0}, \Bvector{}{\cdots}, \Bvector{}{0}]
: {\bm \alpha}\in \mathbb{R}^d\}.
\end{equation*}
In other words,  ${\cal PM}({\cal A})=(E,\widehat{dg})$.

%Observe $f'(F)=dg(F)-1$ for $F\subseteq E$.
%We now show that ${\cal PM}({\cal A}(p))$ can be obtained from $(E,\widehat{dg})$ by Dilworth truncation. 
Denote $V=\{v_1,v_2,\dots, v_n\}$.
For $\bp:V\rightarrow \mathbb{R}^d$, we define a hyperplane $H$ of $\mathbb{P}(V_V)$ by 
\begin{equation*}
H=\{[x_{v_1},x_{v_2},\dots, x_{v_n}] :  {\bm x}_{v}\in V_{v}=\mathbb{R}^d, \mbox{$\sum_{v\in V} \bp(v)\cdot {\bm x}_{v}=0$}\}.
\end{equation*}
Then, observe $A_e(\bp)=A_e\cap H$ for any $e\in E$.
Therefore, if we take $\bp$ so that the set of coordinates of $\bp$ is algebraically independent over $\mathbb{Q}$,
we found that ${\cal PM}({\cal A}(\bp))$ is obtained from ${\cal PM}({\cal A})$ by Dilworth truncation.
By Theorem~\ref{theorem:truncation}, we obtain, for any $F\subseteq E$,
\begin{align*}
\rank(\overline{{\cal A}_F(\bp)})&=\mbox{$\min\{\sum_i(\rank(\overline{{\cal A}_{F_i}})-1) : $ a partition $\{F_1,\dots, F_k\}$ of $F\}$} \\
&=\mbox{$\min\{ \sum_i (\widehat{dg}(F_i)-1) :$ a partition $\{F_1,\dots, F_k\}$ of $F\}$} \\
 &=\mbox{$\min\{\sum_i((\min\{\sum_j dg(F_{i,j}): $ a partition of $F_i \})-1):$ a partition of $F \}$} \\
 &= \mbox{$\min\{\sum_i(dg(F_{i})-1):$ a partition $\{F_1,\dots, F_k\}$ of $F\}$} \\
&=\mbox{$\min\{\sum_i f'(F_{i}):$ a partition $\{F_1,\dots, F_k\}$ of $F\}$}=\hat{f'}(F),
\end{align*}
where we used $f'(F)=dg(F)-1$.
This completes the proof.
\end{proof}

Let $(d-1)\circ G$ be the graph obtained from $G$ by replacing each edge by $(d-1)$ copies, and let $(d-1)\circ E$ be the edge set.
Notice $f'(e)=d-1$ for any $e\in E$.
Hence, applying the same argument given in Lemma~\ref{lemma:circ}, it is not difficult to see that the rank of ${\cal PM}_{f'}(G)=(E,\hat{f'})$ is equal to the rank of 
${\cal M}_{f'}((d-1)\circ G)$, that is, 
the matroid on $(d-1)\circ E$ induced by $f'$.
Thus, Theorem~\ref{theorem:direction} implies a combinatorial characterization of direction-rigidity of bar-joint frameworks proved by Whiteley~\cite{Whitley:1997}.
\begin{corollary}[Whiteley\cite{Whitley:1997}]
\label{corollary}
For almost all joint-configurations $\bp:V\rightarrow \mathbb{R}^d$, 
$(G,\bp)$ is direction-rigid if and only if 
$(d-1)\circ G$ contains an edge subset $I\subseteq (d-1)\circ E$ satisfying the following counting conditions:
\begin{itemize}
\item $|I|=d|V|-(d+1)$;
\item $|F|\leq d|V(F)|-(d+1)$ for any nonempty  $F\subseteq E$. 
\end{itemize}
\end{corollary}

Servatius and Whiteley~\cite{Servatius:1999} further proved a combinatorial characterization of generic rigidity of two-dimensional bar-joint frameworks having both length and direction constraints.
It can be observed that the representation of the associated rigidity matrix can be obtained from the representation of the union of two copies of the graphic matroid by 
restricting  some of rows to a generic hyperplane $H$ and the others to a hyperplane (determined by $H$).
It is still unclear why Theorem~\ref{theorem:truncation} can be extended in this situation.

\section*{Acknowledgments}
The work was supported by Grant-in-Aid for JSPS Research Fellowships for Young Scientists.

\bibliographystyle{abbrv}
\bibliography{tani20120410}

\appendix
\section{Description of Bar-constraints}
\label{app:bar_description}
Here we give a note on how to obtain bar-constraints (\ref{eq:bar_constraint}). 
This note also appears in~\cite[Appendix]{rooted_tree}.

We can coordinatize the exterior product $\mathbb{R}^{d}\wedge \mathbb{R}^d$ as follows:
For $a=(a_1,a_2,\dots, a_d)\in \mathbb{R}^d$ and $b=(b_1,b_2,\dots, b_d)\in \mathbb{R}^d$, 
\begin{equation}
a\wedge b= 
\Bigg(
\Bvector{(1,2)}{\begin{vmatrix} a_1 & a_2 \\ b_1 & b_2 \end{vmatrix}}, 
\Bvector{(1,3)}{-\begin{vmatrix}a_1 & a_3 \\ b_1 & b_3 \end{vmatrix}}, 
\Bvector{}{\ \ \cdots \ \ },  
\Bvector{(i,j)}{(-1)^{i+j+1}\begin{vmatrix}a_i & a_j \\ b_i & b_j \end{vmatrix}}, 
\Bvector{}{\ \ \cdots \ \ }, 
\Bvector{(d-1,d)}{\begin{vmatrix}a_{d-1} & a_d \\ b_{d-1} & b_d \end{vmatrix}} 
\Bigg)\in \mathbb{R}^{{d \choose 2}}.
\end{equation}

\vspace{\baselineskip}

Suppose we are given
rigid bodies $B_1$ and $B_2$ in $\mathbb{R}^d$, which can be identified
with a pair $(p_i,M_i)$ of a point $p_i\in \mathbb{R}^d$ and an orthogonal matrix $M_i\in SO(d)$ for each $i=1,2$.
Namely, each $(p_i,M_i)$ is a local Cartesian coordinate system for each body.
We consider a situation, where the bodies $B_1$ and $B_2$ are connected by a bar.
We denote the endpoints of the bars by $p_1+M_1q_1$ and $p_2+M_2q_2$, where $q_i$ is the coordinate of each endpoint (joint) in the coordinate system 
of each body.

The constraint by the bar can be written by
\begin{equation}
\langle p_2+M_2q_2-p_1-M_1q_1,
p_2+M_2q_2-p_1-M_1q_1\rangle
=\ell^2
\end{equation}
for some $\ell\in \mathbb{R}$.
If we take the differentiation with variables $p_i$ and $M_i$, we get 
\begin{equation}
\langle p_2+M_2q_2-p_1-M_1q_1,
\dot{p}_2+{\dot M}_2q_2-\dot{p}_1-\dot{M}_1q_1 \rangle   
=0
\end{equation}
We may simply assume $p_i=0$ and $M_i=I_d$. Then by setting $h=q_2-q_1$ and $\dot{M}_i=A_i$ with a skew-symmetric matrix $A_i$,
\begin{equation}
\label{eq:const}
\langle h , \dot{p}_2+A_2q_2-\dot{p}_1-A_1q_1\rangle =0.
\end{equation}

Also we denote a skew-symmetric matrix $A$ by  
\begin{equation}
A=\begin{pmatrix} 
0 & -w_{1,2} & \cdots & \cdots & \cdots & \cdots & (-1)^{d+1}w_{1,d} \\ 
w_{1,2} & 0 &  &   &       &        & \vdots \\
 &    &         &   &       &        & \\
\vdots &    &  \ddots       &         & (-1)^{i+j}w_{i,j} &  & \vdots \\
 &    &         &   &       &        & \\
\vdots &    &   &     0    &       &        & \vdots \\
 &    &         &   &       &        & \\
\vdots &    & (-1)^{i+j+1}w_{i,j} &   &   \ddots     &    &  \vdots \\ 
 &    &         &   &       &        & \\
\vdots       &    &         &         &       &     0   & w_{d-1,d} \\
(-1)^{d}w_{1,d}       & \cdots    &   \cdots      &  \cdots       &  \cdots     &  -w_{d-1,d}      & 0   \\
\end{pmatrix}
\end{equation}
and let  $w=\begin{pmatrix}w_{1,2} & w_{1,3} & \cdots & w_{d-1,d}\end{pmatrix}\in \mathbb{R}^{{d \choose 2}}$.
Then, for any $h\in \mathbb{R}^d$ and $q\in \mathbb{R}^d$, we have
\begin{equation}
\label{eq:1}
\langle h, Aq \rangle=\langle q\wedge h, w \rangle.
\end{equation}
Therefore, we can simply describe the infinitesimal bar-constraint (\ref{eq:const}) by 
\begin{equation}
\label{eq:const2}
\langle q_2-q_1, \dot{p}_2-\dot{p}_1\rangle + \langle q_2\wedge q_1, w_2-w_1 \rangle =0,
\end{equation}
where $w_1\in \mathbb{R}^{{d\choose 2}}$ and $w_2\in \mathbb{R}^{{d\choose 2}}$ denote the ${d\choose 2}$-dimensional vectors corresponding to $A_1$ and $A_2$, respectively. 

We call a pair $s_i=(w_i,p_i)\in \mathbb{R}^{d\choose 2}\times \mathbb{R}^d$ a {\em screw motion}, which can be identified with a vector in $\bigwedge^{d-1} \mathbb{R}^{d+1}$. 
Using the homogeneous coordinate of $q_i$ in $\mathbb{P}^d$, (\ref{eq:const2}) is written as 
\begin{equation}
\label{eq:const4}
\langle (q_2,1)\wedge (q_1,1)), s_2-s_1\rangle =0,
\end{equation}
where $[(q_2,1)\wedge (q_1,1)]$ is the Pl{\"u}cker coordinate of the corresponding bar.
\end{document}